\documentclass{amsart}
\usepackage[top=1in, bottom=1in, left=1in, right=1in]{geometry}
\usepackage{amsmath}
\usepackage{amssymb}
\usepackage{amsthm}
\usepackage{verbatim}
\usepackage{graphicx}
\usepackage{graphics}
\usepackage{color}
\usepackage{enumerate}
\usepackage{mathrsfs}
\usepackage{booktabs}
\usepackage[toc,page]{appendix}
\usepackage{enumerate}
\usepackage{mathrsfs}
\usepackage{fancyhdr}
\usepackage{amsopn, amstext, amscd,pifont}
\usepackage{amssymb,bm, cite}
\usepackage[all]{xy}
\usepackage{color}
\usepackage{graphicx,hyperref}
\usepackage{float}
\usepackage{listings}
\usepackage{setspace}
\usepackage{url}
\usepackage{paralist}
\usepackage{subcaption}

\usepackage{graphicx,pgf}
\usepackage{float}

\usepackage{fancyhdr}

\theoremstyle{plain}
\newtheorem*{thm}{Main Theorem}
\newtheorem*{lem}{Main Lemma}
\newtheorem{theorem}{Theorem}[section]
\newtheorem{proposition}[theorem]{Proposition}
\newtheorem{lemma}[theorem]{Lemma}

\theoremstyle{definition}
\newtheorem{definition}[theorem]{Definition}
\newtheorem{example}[theorem]{Example}

\graphicspath{{images/}}

\begin{document}
\title{Boundedness of Hyperbolic Components of Newton Maps}
%\\(Draft--not for distribution)
\author{Hongming Nie and Kevin M. Pilgrim}
\address{Department of Mathematics, Indiana University}
 \email{nieh@indiana.edu\\
 pilgrim@indiana.edu}
\date{\today}
\maketitle

\begin{abstract}
We investigate boundedness of hyperbolic components in the moduli space of  Newton maps. For quartic maps, (i) we prove hyperbolic components possessing two distinct attracting cycles each of period at least two are bounded, and (ii) 
we characterize the possible points on the boundary at infinity for some other types of hyperbolic components. For general maps, we prove hyperbolic components whose elements have fixed superattracting basins mapping by degree at least three are unbounded. 
\end{abstract}

\tableofcontents 

\section{Introduction}

In this paper, we study boundedness properties of hyperbolic components of Newton maps.

A rational map is \textit{hyperbolic} if every critical point converges to a (super)attracting cycle under iteration. Hyperbolicity is an open condition invariant under conjugacy.  Conjecturally hyperbolic maps are dense in moduli space, see \cite{Lyubich00, McMullen94}. A connected component of the set of the hyperbolic maps is called a \textit{hyperbolic component}.  Any two maps in the same hyperbolic component are quasiconformally conjugate on a neighborhood of their Julia sets \cite{Kameyama03}. For a hyperbolic component whose elements have connected Julia sets, there is up to M\"obius conjugacy a unique post-critically finite map in this component \cite{McMullen88}.  In terms of the ``dictionary'' between rational maps and Kleinian groups, hyperbolic components are analogous to components of the set of convex compact discrete faithful representations of a fixed group into $\mathrm{PSL}_2(\mathbb{C})$; see \cite{McMullen88, Pilgrim94}. For more properties of hyperbolic components, we refer \cite{Milnor86, Rees90}.\par

Our general focus is on hyperbolic components in algebraic subfamilies $\mathcal{F}$ of the \emph{moduli space} $\mathrm{rat}_d$ of rational maps of degree $d$, which is the quotient of the parameter space $\mathrm{Rat}_d \subset \mathbb{P}^{2d+1}$ of all rational maps of degree $d$ under the conjugation action by $\mathrm{Aut}(\mathbb{P}^1)$. If $\mathcal{H} \subset \mathcal{F} \subset \mathrm{rat}_d$ is a component of hyperbolic maps in $\mathcal{F}$, we say it is \emph{bounded}  in $\mathcal{F}$ if the closure $\overline{\mathcal{H}}\subset \mathrm{rat}_d$ is compact.
\par

The simplest examples of hyperbolic components arise in the quadratic polynomial family $P_c(z)=z^2+c$. In this case, there two kinds of components. The unbounded \emph{escape locus} is the complement of the Mandelbrot set, and has fractal boundary. The bounded components are uniformized by the multiplier $\lambda$ of the unique attracting cycle \cite{Douady84}, so as subsets of the plane they are semi-algebraic, defined as a component of the real-algebraic inequality $|\lambda|<1$. More general structural results regarding hyperbolic components in the spaces of higher degrees polynomials are given in \cite{Milnor12, Roesch07, Wang17}.  \par

\vskip 0.1in

\noindent{\bf Boundedness results.}   For quadratic rational maps, in \cite{Rees90} Rees divided the hyperbolic components  into $4$ types (see also \cite{Milnor93}), and showed boundedness of certain one real dimension loci in hyperbolic components of certain types. Epstein \cite{Epstein00} provided the first general compactness result for hyperbolic components in moduli space of rational maps.  Following Rees' classification, we say a hyperbolic component $\mathcal{H} \subset \mathrm{rat}_2$ is \emph{type D} if its maps possess two distinct attracting cycles. 
Epstein showed that if $\mathcal{H}$ is of type D and the periods of the attractors are each at least two, then $\mathcal{H}$ is bounded in $\mathrm{rat}_2$.  Epstein's argument relies crucially on two facts: (i) a sequence in the moduli space of rational maps is bounded if and only if the multipliers at fixed-points remain bounded, and (ii) there are two distinct cycles of periods at least two whose multipliers remain bounded.  In general degrees, and using very different methods inspired by Cui's extremal length control of distortion technique, Haissinsky and Tan \cite{Haissinsky04} showed a boundedness result similar to Rees' \emph{op. cit.} for  certain simple pinching deformations in wide generality. In \cite{McMullen94c} McMullen conjectured that if a hyperbolic rational map has Julia set homeomorphic to the Sierpinski carpet, then the corresponding hyperbolic component is bounded. \par
\vskip 0.1in

\noindent{\bf Unboundedness results.}  On the other side, Petersen \cite{Petersen99} related the failure of mating to unboundedness of hyperbolic components. Makienko \cite{Makienko00} gave sufficient conditions for the unboundedness of hyperbolic components, also see \cite{Tan02}. These use so-called pinching deformations; cf. \cite{Pilgrim94}. \par
\vskip 0.1in

\noindent{\bf Newton maps.}  Our main results concern boundedness properties of hyperbolic components in the algebraic family of Newton maps. Recall that for a polynomial $P(z)\in\mathbb{C}[z]$ of degree $d\ge 2$ with distinct roots $\{r_1, \ldots, r_d\}$, its Newton map is defined by 
$$N_P(z)=N_{\{r_1, \ldots, r_d\}}(z):=z-\frac{P(z)}{P'(z)}.$$ 
We denote by $\mathrm{NM}_d \subset \mathrm{Rat}_d$ the subspace of Newton maps and by $\mathrm{nm}_d$ the moduli space of affine conjugacy classes of Newton maps.  Since $\mathrm{nm}_d$ is closed in $\mathrm{rat}_d$, a sequence in $\mathrm{nm}_d$ tends to infinity in $\mathrm{nm}_d$ if and only if it tends to infinity in $\mathrm{rat}_d$. \par

Since 
$$N_P'(z)=\frac{P(z)P''(z)}{P'(z)^2},$$
a critical point of $N_P(z)$ is either a zero of $P(z)$, a zero of $P''(z)$ or both. We say a critical point of $N_P$ is \textit{additional} if it is a zero of $P''(z)$, and we say it is \textit{free} if it is not fixed. Note a free critical point is additional.  Thus, a Newton map $N_P(z)$ is hyperbolic if and only if each free critical point is attracted to some (super)attracting cycle.  \par

The space $\mathrm{nm}_2$ is a singleton, a fact known to Cayley \cite{Cayley79}; there are no additional critical points. The space $\mathrm{nm}_3$ is isomorphic to $\mathbb{C}$; there is one additional critical point; there is a unique unbounded hyperbolic component and all other components are bounded with Jordan curves boundaries, see \cite{Roesch15}. In general, the space $\mathrm{nm}_d$ has dimension $d-2$; there are $d-2$ additional critical points. 

\subsection{Main result}
Our first main result is the following theorem, which is an analogue of Epstein's boundedness result for quadratic rational maps. To set up the statement, we say a hyperbolic component $\mathcal{H}\subset \mathrm{nm}_4$ is \emph{type D} if there are two distinct (super)attracting cycles each of period at least two; each then attracts a free critical point. We call these cycles the \emph{free (super)attracting cycles} of elements of $\mathcal{H}$. 

\begin{thm}\label{Thm-bounded}
A hyperbolic component $\mathcal{H}\subset \mathrm{nm}_4$ of type D is bounded in $\mathrm{nm}_4$. 
\end{thm}
In the proof of the Main Theorem, instead of degenerating sequences in Epstein's argument \cite{Epstein00},  we study degenerating holomorphic (one-complex-dimensional) families: holomorphic maps $\mathbb{D}\to \overline{\mathrm{Rat}}_d$ given by $t\to f_t$ with $f_t\in\mathrm{Rat}_d$ if $t\not=0$. The fact that we can reduce sequences to such families comes from the semi-algebraicity of the component $\mathcal{H}$. In any algebraic family, Milnor announced that a hyperbolic component is semi-algebraic if it possesses the maximal number of distinct (super)attracting cycles. For a more general version of his result, we refer \cite{Milnor14}. In our setting, if $\mathcal{H}$ is type D, then $\mathcal{H}$ is semi-algebraic and hence semi-analytic. Our Main Theorem is a consequence of the following Main Lemma.

\begin{lem}\label{main}
Suppose $\mathcal{H}\subset \mathrm{nm}_4$ is of type D.  Suppose $\{N_t\}$ is a holomorphic family of quartic Newton maps with the property that there exists a sequence $\{t_k\}$ with $t_k\to 0$ as $k\to\infty$ such that $[N_{t_k}]\in\mathcal{H}$. Then $[N_t]$ is bounded, as $t \to 0$, in $\mathrm{nm}_4$.
\end{lem}
\begin{proof}[Proof of the Main Theorem] 

We first lift from moduli to parameter space.  For distinct complex $r,s$ with $|r|\leq 1, |s|\leq 1$, and neither $r$ nor $s$ equal to $0$ or $1$, let $N_{(0,1,r,s)}$ denote the Newton map for the polynomial $z(z-1)(z-r)(z-s)$ with distinct and marked roots $\{0,1,r,s\}$, and let $\mathrm{nm}_4^\# \subset \overline{\mathbb{D}}\times\overline{\mathbb{D}}$ be the set of such Newton maps with marked roots. There is a natural surjective map $p: \mathrm{nm}_4^{\#} \twoheadrightarrow \mathrm{nm}_4$.  

Now suppose $\mathcal{H} \subset \mathrm{nm}_4$ is a hyperbolic component of type D, and let $\mathcal{H}^\#$ be a component of $p^{-1}(\mathcal{H})$.  A result of Milnor \cite{Milnor14} implies that $\mathcal{H}^\#$ is a semi-algebraic bounded domain in $\mathbb{C}^2$; see Lemma \ref{lemma:semialg} below.  

Now suppose $x$ is a boundary point of $\mathcal{H}^\#$ in $\overline{\mathbb{D}}\times \overline{\mathbb{D}}$.  By the semi-analytic version of the Curve Selection Lemma \cite{Lojasiewicz95} (see \cite{Denkowska15} for a survey),  there is an analytic function $\kappa:[0,1]\to\mathbb{R}^4=\mathbb{C}^2$ such that $\kappa((0,1])\subset\mathcal{H}^\#$ and $x=\kappa(0)\notin\mathcal{H}^\#$. Write $\kappa(t)=(\kappa_1(t),\kappa_2(t), \kappa_3(t), \kappa_4(t))$ for $t\in[0,1]$. Then $\kappa$ extends to a holomorphic function $\kappa:\mathbb{D}\to\mathbb{C}^2 \subset \overline{\mathrm{Rat}}_4$, sending $t\in\mathbb{D}$ to $(\kappa_1(t)+i\kappa_2(t), \kappa_3(t)+i\kappa_4(t))$. We obtain a holomorphic family $N_t$ of Newton maps for the polynomials $z(z-1)(z-(\kappa_1(t)+i\kappa_2(t)))(z-(\kappa_3(t)+i\kappa_4(t)))$. 

The hypotheses of the Main Lemma are satisfied, so $[N_t]$ is bounded, as $t\to 0$, in $\mathrm{nm}_4$. It follows that $\mathcal{H}$ is bounded in $\mathrm{nm}_4$. 
\end{proof}
\begin{figure}[h!]
  \includegraphics[width=.75\linewidth]{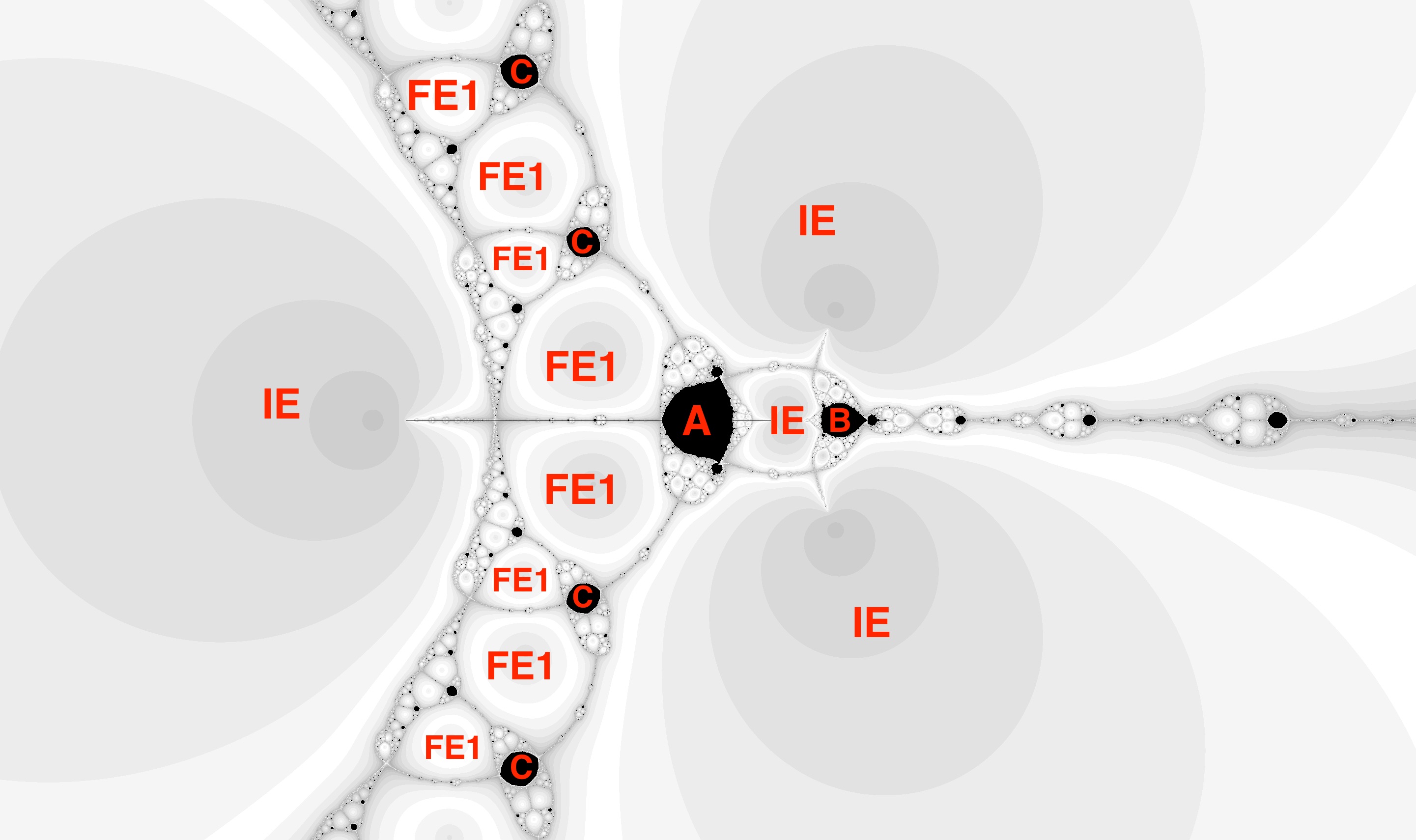}
  \caption{The locus $\mathrm{Per_2(0)}\cap\mathrm{nm}_4$, showing part of $c$-plane for the family of Newton maps for the polynomial $z^4/12-cz^3/6+(4c-3)z/12+(3-4c)/12$, with additional critical points $0$ and $c$. The periodic critical orbit is $0\to 1\to 0$. The letters indicate the types of hyperbolic components, see section \ref{type}. This and subsequent images were created using ``FractalStream", see \cite{FractalStream}.}
  \label{per_2}
\end{figure}
\begin{figure}[h!]
 \includegraphics[width=.75\linewidth]{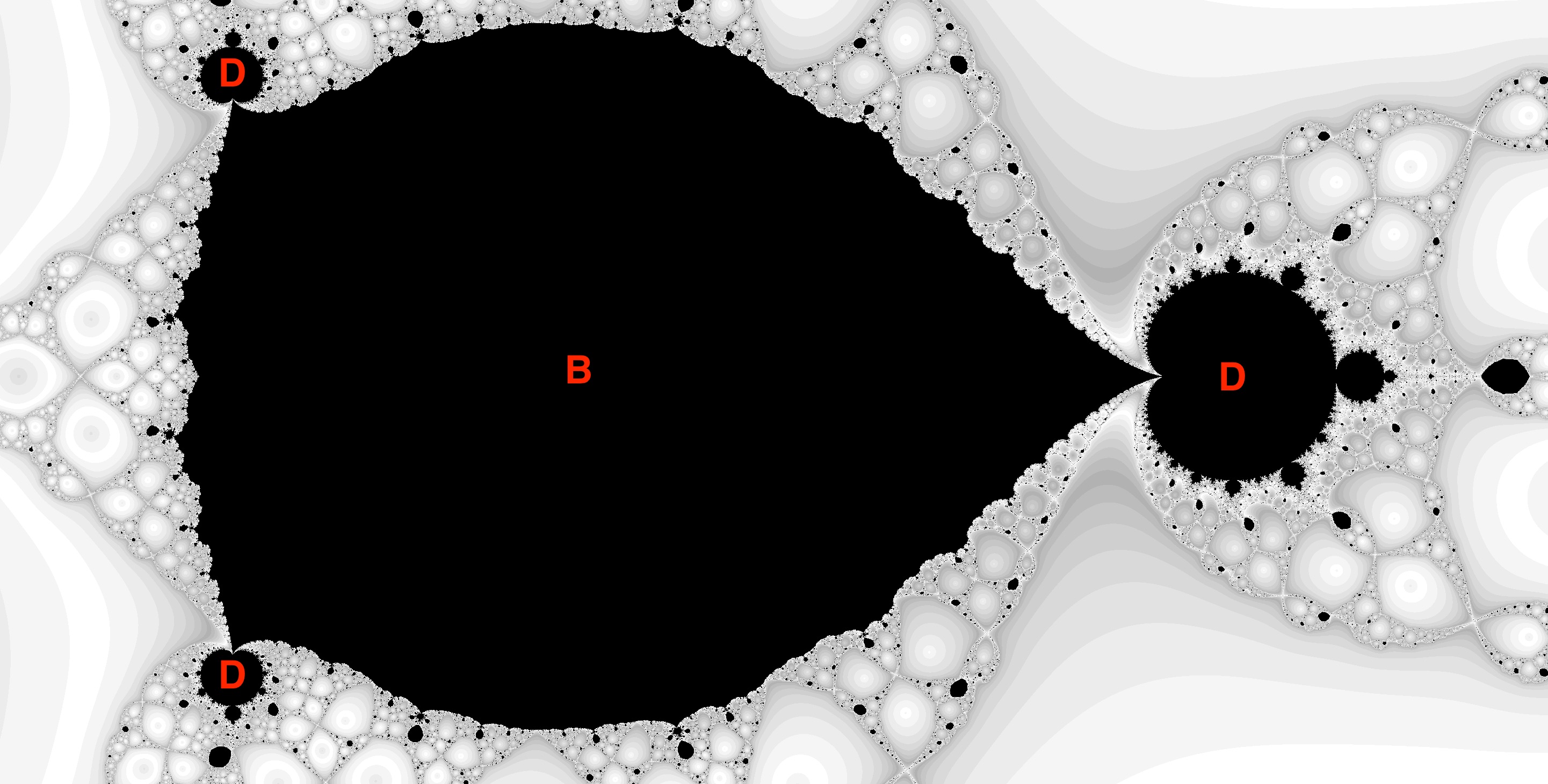}
  \caption{The part of locus $\mathrm{Per_2(0)}\cap\mathrm{nm}_4$ by zooming in the type B component indicated in Figure \ref{per_2}. The Main Theorem asserts that the hyperbolic components in $\mathrm{nm}_4$ containing a part indicated by D are bounded. The component labelled B has ``fractal'' boundary while those labelled D have real algebraic boundaries.}
\end{figure}

\subsection{Idea of proof of the Main Lemma} Our overall strategy mimics that of Epstein \cite{Epstein00} and proceeds by contradiction. We begin by assuming the existence  of a holomorphic family $N_t$ such that $[N_t]$ meets $\mathcal{H}$ in a degenerating sequence $N_{t_k}$.  We then show the existence of a suitable rescaling limit $g=\lim\limits_{k\to\infty}M_k^{-1}\circ N_{t_k}^{\circ q}\circ M_k$.  Using both analytical and geometric estimates, we show that the critical orbit dynamics of $g$ are constrained. Finally, we use Epstein's refined Fatou-Shishikura inequality \cite{Epstein99} to derive further constraints on the critical orbit dynamics of the limit, leading to a contradiction.

Our arguments here replace some of the delicate analytical estimates in Epstein's proof \cite{Epstein00} with a softer analysis using Berkovich dynamics. This fruitful idea was used by Kiwi in \cite{Kiwi15} to investigate  rescaling limits for one-parameter holomorphic families. Later it was used in DeMarco-Faber's work on the weak limits of measures of maximal entropy under degenerating sequences and families, see \cite{DeMarco14, DeMarco16}. Inspired by each of these investigations, our present work is based on  recent results on degenerating Newton maps obtained by the first author in \cite{Nie-thesis}.  For details of what follows in the remainder of this subsection, see \cite{Kiwi06,Kiwi14} and section \ref{Berk} below.  

Let $\{f_t\}$ be a holomorphic family, let $\mathbb{L}$ be the completion of the field of formal Puiseux series in the indeterminate $t$ over $\mathbb{C}$, and let $\mathbf{P}^1$ be the corresponding Berkovich space. It turns out that the family $\{f_t\}$ induces a map $\mathbf{f}:\mathbf{P}^1\to\mathbf{P}^1$ and the dynamical system $(\mathbf{f}, \mathbf{P}^1)$ can be regarded as encoding the possible limits of the complex dynamical systems $(f_t, \mathbb{P}^1)$ at scales and locations varying in $t$. \par
Returning the discussion back to Newton maps, a holomorphic family $\{N_t\}$ of Newton maps induces a Newton map $\mathbf{N}\in\mathbb{L}(z)$ acting on the Berkovich space $\mathbf{P}^1$. The Berkovich dynamics of $\mathbf{N}$ give restrictions on the limiting locations of the attracting cycles of $N_t$ which in turn constrain the limiting behavior of orbits of critical points.
Another ingredient, which uses only complex analysis, is the estimation of the sizes of some immediate basins of the attracting cycles; see Proposition \ref{Fatou-size}, which is an analogue of \cite[Proposition 5]{Epstein00}.  

\subsection{Boundaries of hyperbolic components and tameness} The proof of Main Theorem exploits the fact that the boundaries of such hyperbolic components are not too pathological.  Our techniques allow us to obtain some partial results regarding boundedness of other types of hyperbolic components. 

A rational map in the boundary of a hyperbolic component has either a critical point in Julia set or an indifferent cycle. The boundaries of hyperbolic components, and hence the hyperbolic components themselves, may have complicated structure, see \cite{Bonifant18} for the hyperbolic components of antipode preserving cubic rational maps. In \cite{Milnor14}, Milnor conjectured that if there is a rational map having no indifferent cycles in the boundary of some hyperbolic component, then this boundary is a fractal set. However, in the complementary case, Milnor conjectured that the boundaries of corresponding hyperbolic components are in fact semi-algebraic.

To formulate our partial results, we need a very weak notion of accessibility of a boundary point. 
\begin{definition}
For $n\ge 1$, suppose $\mathcal{U}\subset \mathbb{P}^{n}$ is an analytic space and $x \in \partial \mathcal{U}$. 
We say the pair $(\mathcal{U},x)$ is \emph{holomorphically sequentially accessible} if there is a holomorphic map $\phi:\mathbb{D}\to\mathbb{P}^{n}$ and a sequence $\{t_k\}\subset\mathbb{D}$ with $t_k\to 0$ as $k\to\infty$ such that $\phi(0)=x$ and $\phi(t_k)\in\mathcal{U}$. 
\end{definition}
This is much weaker than other notions of accessibility. For example, in $\mathbb{P}^1=\mathbb{C}\cup \{\infty\}$, let $\mathcal{U}= \mathbb{C}\setminus \{ se^{ie^s} : s\geq 0\}$.  Then $(\mathcal{U}, \infty)$ is holomorphically sequentially accessible via the family $\phi(t)=1/t$. On the other hand, consider $\mathcal{V}=\{(x,y) \in \mathbb{C}^2: e^{-1/|x|^2}<|y|<2 e^{-1/|x|^2}\}$. The  pair $(\mathcal{V}, 0)$ is not holomorphically sequentially accessible. Using similar ideas it is easy to further arrange that $\mathcal{V}$ is simply-connected. \par

To set up the next definition, let $\mathcal{H} \subset \mathrm{rat}_d$ be a hyperbolic component, and $\widetilde{\mathcal{H}}$ its lift to $\mathrm{Rat}_d$. Recall that we have a GIT compactification  $\mathrm{rat}_d\subset \overline{\mathrm{rat}}^{\mathrm{GIT}}_d$, which is the categorical quotient of the semistable points in $\overline{\mathrm{Rat}}_d=\mathbb{P}^{2d+1}$, see \cite{Silverman98}.  

\begin{definition}[Tame boundary point]
\label{def:tame}
A boundary point $[f] \in \partial{\mathcal{H}}$ of a hyperbolic component $\mathcal{H}\subset \overline{\mathrm{rat}}^{\mathrm{GIT}}_d$ is a \emph{tame} boundary point if (i) there exists a point $f_0 \in \partial{\widetilde{\mathcal{H}}}\subset \overline{\mathrm{Rat}}_d$ such that $(\widetilde{\mathcal{H}}, f_0)$ is holomorphically sequentally accessible via a family $f_t$, and (ii) for some sequence $t_k \to 0$, $[f_{t_k}] \to [f]$. Otherwise, the boundary point $[f]$ of $\mathcal{H}$ is \emph{wild}. 
\end{definition}
Note that in the definition we do not require that $f_0$ be semistable, i.e. that $[f]=[f_0]$ in the GIT compactification. 

\begin{theorem}\label{accessible-bounded}
Let $\mathcal{H}\subset \mathrm{nm}_4$ be a hyperbolic component of quartic Newton maps having a free (super)attracting cycle whose grand orbit basin contains both additional critical points. If $[N] \in \partial{\mathcal{H}} \cap \partial\mathrm{nm}_4$, then either
\begin{enumerate}
\item $[N]$ is wild, or 
\item $[N]=[N_0]$ is tame, where in projective coordinates, 
$$N_0([X:Y])=YN_{\{0, 1, r\}}([X:Y])$$ 
for $r\in\mathbb{C}\setminus\{0,1\}$ such that $N_{\{0, 1, r\}}$ has a nonfixed (super)attracting, irrationally indifferent or parabolic-repelling cycle. 
\end{enumerate}
\end{theorem}
In other words, if $\mathcal{H}$ is unbounded, then as one approaches a boundary point at infinity, $\mathcal{H}$ either looks geometrically very weird, or the dynamics degenerates in a very specific way that our techniques cannot rule out. 
The proof proceeds as follows. If $[N] \in\partial{\mathcal{H}} \cap \partial\mathrm{nm}_4$ is tame, then by definition, as in the proof of Theorem \ref{Thm-bounded}, we can find a holomorphic family degenerating along a sequence in $\mathcal{H}$. We again analyze the induced Berkovich dynamics and use an analog  of the Main Lemma. \par

There is a classification of hyperbolic components in $\mathrm{nm}_4$ which is somewhat analogous to that for quadratic rational maps; see Section \ref{type}. Theorem \ref{accessible-bounded} asserts that a hyperbolic component $\mathcal{H}\subset\mathrm{nm}_4$ of type indicated by A, B or C in Figure \ref{per_2} is either (i) bounded in $\mathrm{nm}_4$, or (ii) its boundary at infinity consists of wild points and/or tame points of the above specific type.

In the other direction, we obtain unboundedness results for components we call type IE, which stands for ``immediate escape''; see Section \ref{type}. Using quasiconformal deformations and ideas of Petersen-Ryd \cite{Petersen00}, we show
\begin{theorem}\label{escape-unbounded}
Suppose $\mathcal{H}\subset \mathrm{nm}_d$ is a hyperbolic component consisting of maps for which some additional critical point lies in the immediate basin of a superattracting fixed point. Then $\mathcal{H}$ is unbounded. 
\end{theorem}
Our use of the term ``escape'' here is  different from that used to describe both hyperbolic components for quadratic rational maps in \cite{Milnor93} and cubic polynomials in \cite{Bonifant10}.  

\subsection{Outline}
In section \ref{pre}, we give the complex analytic background. We also prove some properties of the reduction of the induced map for a degenerate family of quartic Newton maps, and give a classification of hyperbolic components  in $\mathrm{nm}_4$. In section \ref{Berk}, we summarize the structure results regarding Berkovich dynamics of quartic Newton maps used in this work. In section \ref{cycle}, we identify the limits of the distingiushed attracting cycles. Then we prove the Main Lemma in section \ref{compact} and Theorem \ref {accessible-bounded} in section \ref{accessible}. Finally, we prove Theorem \ref{escape-unbounded} in section \ref{qc}.

\subsection{Acknowledgements} Kevin Pilgrim was supported by Simons Collaboration Grant \# 245269.

\section{Complex analytic preliminaries}\label{pre}
\subsection{Rational maps}
Parametrized by the coefficients, the space $\mathrm{Rat}_d$ of degree $d$ rational maps can be identified with an open dense subset of $\mathbb{P}^{2d+1}$. For each point $f\in\mathbb{P}^{2d+1}$, there exists two homogenous degree $d$ polynomials $F_a(X,Y)$ and $F_b(X,Y)$ such that 
$$f([X:Y])=[F_a(X,Y):F_b(X,Y)].$$
Set $H_f=\gcd(F_a,F_b)$. Then we can write $f=H_f\hat f$, where $\hat f$ is a rational map of degree at most $d$. A zero of $H_f$ is called a \textit{hole} of $f$. We say a point $f\in\mathbb{P}^{2d+1}$ is a \textit{degenerate rational map} if it has a hole. Denote $\mathrm{Hole}(f)$ the set of all holes of $f$.
\begin{lemma}\cite[Lemma 4.1]{DeMarco05}\label{convergence}
Suppose that a sequence $\{f_k\}\subset\mathrm{Rat}_d$ converges to $f=H_f\hat f\in\mathbb{P}^{2d+1}$. Then $f_k$ converges to $\hat f$ locally uniformly outside $\mathrm{Hole}(f)$.
\end{lemma} 
For $t\in\mathbb{D}$, a collection $\{f_t\}\subset\mathbb{P}^{2d+1}$ is a \textit{holomorphic family of degree $d\ge 1$ rational maps} if the map $\mathbb{D}\to\mathbb{P}^{2d+1}$, sending $t$ to $f_t$, is a holomorphic map and $f_t\in\mathrm{Rat}_d$ if $t\not=0$. We say $\{f_t\}$ is \textit{degenerate} if $f_0\not\in\mathrm{Rat}_d$. For the version of Lemma \ref{convergence} for holomorphic families, we refer \cite[Lemma 3.2]{Kiwi15}.\par
The moduli space $\mathrm{rat}_d:=\mathrm{Rat}_d/\mathrm{PSL}_2(\mathbb{C})$, modulo the action by conjugation of the group of M\"obius transformations, is a complex orbifold of dimension $2d-2$. For quadratic case, the space $\mathrm{rat}_2$ is biholomorphic to $\mathbb{C}^2$ \cite[Lemma 3.1]{Milnor93}.\par

\subsection{Newton maps} 
Fix $d \geq 2$, and suppose $(r_1, \ldots, r_d) \in \mathbb{C}^d$ are pairwise distinct. Let $P:=(z-r_1)\cdot \ldots \cdot (z-r_d)$.  We denote the corresponding Newton map by 
\[ N_{\{r_1, \ldots, r_d\}}(z):=N_P(z) \in \mathrm{NM}_d.\]
It is often convenient to work with marked maps. We denote by $\mathrm{NM}_d^* \subset \mathrm{Rat}_d\times \mathbb{C}^d$ the set of marked maps $N_{(r_1, \ldots, r_d)}(z)$; note the subscript is a list, not a set. 

The map $N _P$ also has degree $d$ and each root of $P$ is a superattracting fixed point of $N_P$. The only other fixed point of $N_P$ is at $\infty$ with multiplier $d/(d-1)$. Hence if two Newton maps are M\"obius conjugate, then the conjugacy is affine. Thus the moduli space of degree $d$ Newton maps is naturally defined by $\mathrm{nm}_d:=\mathrm{NM}_d/\mathrm{Aut}(\mathbb{C})$, which is a subset of $\mathrm{rat}_d$. Since Newton maps are uniquely determined by the roots of the corresponding polynomials, the moduli space $\mathrm{nm}_d$ has complex dimension $d-2$.
The moduli space of marked maps $\mathrm{nm}_d^*:=\mathrm{NM}_d^*/\mathrm{Aut}(\mathbb{C})$ is canonically isomorphic to $\mathbb{C}^{d-2}-\Delta$, where $\Delta=\{r_i=0, r_j=1, r_k=r_l (k\neq l)\}_{i,j,k,l \in \{1, \ldots, d-2\}}$.  \par
A family of Newton maps tends to infinity in parameter space $\mathrm{NM}_d$ if and only if some roots collide and/or some roots tend to infinity. Using homogeneous coordinates, we may extend the notation $N_{\{r_1, \ldots, r_d\}}$ to describe such degenerate maps. 

\begin{example}
Consider the cubic Newton map $N_{\{0,1,t\}}(z)$ for the polynomial $z(z-1)(z-t)$.  If $t\to 0$, in projective coordinates we get the degenerate Newton map with a hole at $z=0$
$$N_0([X:Y])=X[2X^2-XY:3XY-2Y^2]=:N_{\{0,0,1\}}([X:Y]).$$
In this case, we have $\widehat{N}_0$ has degree two, a superattracting fixed point at $1$, and an attracting fixed point at $0$ with multiplier $1/2$.

If $t\to\infty$, we get the degenerate Newton map with a hole at $z=\infty$
$$N_\infty([X:Y])=Y[X^2:2XY-Y^2]=:N_{\{0,1,\infty\}}([X:Y]).$$
In this case, $\widehat{N}_\infty=N_{\{0,1\}}$ is a degree two nondegenerate Newton map and has two superattracting fixed points at $0$ and $1$. 

This example is interesting since the first family $N_{\{0, t, 1\}}$ is conjugate to the second family $N_{\{0, 1, 1/t\}}$ via the affine family $M_t(z)=z/t$. 
\end{example}

\subsection{Lifting to marked normalized maps}\label{lifting}

By conjugating via an affine map so that a pair of roots maximizing the planar distance between the $d$ roots is sent to the pair $\{0, 1\}$, we obtain a finite-to-one surjective map 
\[ p: \mathrm{nm}_d^{\#}:=\overline{\mathbb{D}}^{d-2}-\Delta \twoheadrightarrow \mathrm{nm}_d\]
given by 
\[ (r_3, \ldots, r_d) \mapsto [N_{\{0, 1, r_3,  \ldots, r_d\}}].\]
Note that $p$ factors as a composition $\mathrm{nm}_d^\# \hookrightarrow \mathrm{NM}_d^* \twoheadrightarrow \mathrm{nm}_d$.  

\begin{lemma}
\label{lemma:semialg}
Suppose $\mathcal{H}\subset \mathrm{nm}_d$ is a hyperbolic component consisting of maps with $d-2$ distinct free (super)attracting cycles.
%each of period at least $2$. 
Then $p^{-1}(\mathcal{H})\subset \overline{\mathbb{D}}^{d-2}$ is semi-algebraic.  It consists of finitely many connected open bounded sets $\mathcal{H}^\#$, each of which is semi-algebraic. 
\end{lemma}

\begin{proof} Milnor \cite{Milnor14} shows that in any complex algebraic family $\mathcal{F} \subset\overline{\mathrm{Rat}}_d :=\mathbb{P}^{2d+1}$ of complex dimension $m$, a component $\widetilde{\mathcal{H}}$ of the locus where $m$ distinct attracting cycles have multipliers less than one in modulus is a semi-algebraic set.  Applying this to the $m:=d-2$-complex-dimensional algebraic family $\mathcal{F}$ of maps $N_{\{0,1, r_3, \ldots, r_d\}} \subset \mathrm{Rat}_d$, we conclude that such hyperbolic components $\widetilde{\mathcal{H}} \subset \mathcal{F}$ are semi-algebraic in $\mathbb{P}^{2d+1}$. The natural quotient map $q: \mathbb{C}^{d-2}-\Delta \to \mathrm{NM}_d\subset \mathbb{P}^{2d+1}$ given by $(r_3, \ldots, r_d) \mapsto N_{\{0,1,r_3, \ldots, r_d\}}$ is a polynomial map, hence the preimage $\widetilde{\widetilde{\mathcal{H}}}:=q^{-1}(\widetilde{\mathcal{H}})\subset \mathbb{C}^{d-2}$ is semi-algebraic. The locus $\mathrm{nm}_d^\# \subset \mathbb{C}^{d-2}$ is defined by a collection of real-algebraic polynomial equalities and inequalities of the form  
\[ r_i\neq 0, r_j \neq 1, r_k \neq r_l (k\neq l),  |r_n|^2\leq 1 \;\; (i,j,k,l,n \in \{3, \ldots, d\}) \]
and so is semi-algebraic. An intersection of semi-algebraic sets is semi-algebraic  (see \cite{Bochnak98}).  Hence $\widetilde{\widetilde{H}}\cap \mathrm{nm}_d^\#=p^{-1}(\mathcal{H})$ is semi-algebraic. The natural projection $\mathrm{nm}_d^\# \to \mathrm{nm}_d$ is continuous and $\mathcal{H}$ is open, hence $p^{-1}(\mathcal{H})$ is open.  A semi-algebraic set has finitely many connected components, each of which is semi-algebraic (\cite[Proposition 2.8]{Bochnak98}), and the Lemma follows.

\end{proof}

We now specialize to quartic maps and first roughly classify the possible degeneracies. Suppose $(r_k, s_k) \in\mathrm{nm}_4^{\#}$ converges to $(r,s) \in \Delta \cap \partial  \mathrm{nm}_4^{\#}$.  Up to conjugacy, there are precisely three qualitatively distinct possibilities:
\begin{enumerate}
\item \textit{type 1} $0=r\neq s\neq 1$; 
\item \textit{type 2} $0=r\neq s=1$;
\item \textit{type 3} $0=r=s\neq 1$.
\end{enumerate}
In fact, if $(r(t),s(t))$ is holomorphic in $t$ and defines a degenerating family $N_t$, we will further distinguish two subcases 3a, 3b in type 3 according to the relative rates at which $(r(t),s(t)) \to (0,0)$ (Lemma \ref{types}). \par

\subsection{Epstein's invariants}  To control the dynamical possibilities of rescaling limits, we will use Epstein's refined version of the Fatou-Shishikura inequality (FSI) \cite{Epstein99}, presented here in this section. 

Let $f$ be an analytic map on $U\subset\mathbb{C}$ fixing $z\in U$. Then the multiplier $\rho_z(f)$ is defined by $\rho_z(f)=f'(z)$. If $f$ is not the identity, then the \textit{topological multiplicity} $\mathfrak{m}_z(f)$ of $f$ at $z$ is defined by 
$$\mathfrak{m}_z(f)=\frac{1}{2\pi i}\int_{\mathcal{C}}\frac{1-f'(z)}{z-f(z)}dz$$
and the \textit{holomorphic index} is defined by
$$\mathfrak{i}_z(f)=\frac{1}{2\pi i}\int_{\mathcal{C}}\frac{1}{z-f(z)}dz,$$
where  $\mathcal{C}$ is any sufficiently small, positively-oriented, simple closed curve around $z$. These quantities are invariant under holomorphic change of coordinates, and the topological multiplicity $\mathfrak{m}_z(f)$ is always a positive integer. Note that $\mathfrak{m}_z(f)=1$ if and only if $\rho_z(f)\not=1$, and in this case
$$\mathfrak{i}_z(f)=\frac{1}{1-\rho_z(f)}.$$\par
The Cauchy's integral formula implies
$$\sum_{z=f(z)\in V}\mathfrak{m}_z(f)=\frac{1}{2\pi i}\int_{\partial V}\frac{1-f'(z)}{z-f(z)}dz$$
and 
$$\sum_{z=f(z)\in V}\mathfrak{i}_z(f)=\frac{1}{2\pi i}\int_{\partial V}\frac{1}{z-f(z)}dz$$
for any open set $V\Subset U$ with rectifiable boundary containing no fixed points. Thus if $f:\mathbb{P}^1\to\mathbb{P}^1$ be a rational map of degree $d$, it follows from the residue theorem that
$$\sum_{z=f(z)\in\mathbb{P}^1}\mathfrak{m}_z(f)=d+1$$
and 
$$\sum_{z=f(z)\in\mathbb{P}^1}\mathfrak{i}_z(f)=1.$$\par
We say a fixed point $z$ of the map $f$ is \textit{superattracting}, \textit{attracting}, \textit{indifferent} or \textit{repelling} according to $\rho_z(f)=0$, $0<|\rho_z(f)|<1$, $|\rho_z(f)|=1$ or $|\rho_z(f)|>1$. If $|\rho_z(f)|\not=1$ or $\rho_z(f)=1$, we have $\mathfrak{m}_z(f^n)=\mathfrak{m}_z(f)$ for any $n\ge 1$. When studying the case $\rho_z(f)=1$, it is convenient to work on the \textit{r\'esidu it\'eratif} $\mathfrak{I}_z(f)$ defined by 
$$\mathfrak{I}_z(f)=\frac{\mathfrak{m}_z(f)}{2}-\mathfrak{i}_z(f).$$
In this case for all $n\not=0$, we have 
$$\mathfrak{I}_z(f)=n\mathfrak{I}_z(f^n).$$\par
For a fixed point $z$ of the map $f$, if $\rho_z(f)=e^{2\pi ip/q}$, where $(p,q)=1$, and $f^q$ is not identity, then there is an integer $\nu$ such that $\mathfrak{m}_z(f^q)=\nu q+1$. We say this fixed point $z$ is \textit{parabolic} and its \textit{degeneracy} is $\nu$. Furthermore, we say $z$ is \textit{parabolic-attracting}, \textit{parabolic-indifferent} or \textit{parabolic-repelling} according to $\mathrm{Re}(\mathfrak{I}_z(f))$ is negative, zero or positive.\par
For a degree $d\ge 1$ rational map $f:\mathbb{P}^1\to\mathbb{P}^1$, we say $z$ is a \textit{periodic point} if $f^n(z)=z$ for some $n\ge 1$. The smallest such $n$ is called the \textit{period} of $z$.  Denote $\langle z\rangle$ the cycle $\{z, f(z),\dots, f^{n-1}(z)\}$ and denote $z^{(i)}=f^i(z)\in\langle z\rangle$. The multiplier, multiplicity and index of the cycle $\langle z\rangle$ are the corresponding invariants of $z$ as a fixed point of $f^n$. The Fatou-Shishikura inequality asserts that $f$ has at most $2d-2$ nonrepelling cycles, see \cite{Shishikura87}. Now associate to each cycle $\langle z\rangle$ the \textit{cycle invariant} 
$$\gamma_{\langle z\rangle}=
\begin{cases}
0&\text{if}\  \langle z\rangle\ \text{is\  repelling\ or\ superattracting},\\
1&\text{if}\  \langle z\rangle\ \text{is\  attracting\ or\ irrationally indifferent},\\
\nu &\text{if}\  \langle z\rangle\ \text{is\  parabolic-repelling},\\
\nu+1& \text{if}\  \langle z\rangle\ \text{is\  parabolic-attracting\ or\ parabolic-indifferent}.
\end{cases}$$
We now define two invariants of the map $f$ itself. Let  
$$\gamma(f)=\sum_{\langle z\rangle\subset\mathbb{P}^1}\gamma_{\langle z\rangle}.$$
Finally, let $\delta(f)$ denote the number of infinite tails of critical orbits. Obviously, $\delta(f)\le 2d-2$. The following result, due to Epstein, refines the FSI.

\begin{proposition}[Refined FSI]\cite[Theorem 1]{Epstein99}\label{number}
Let $f$ be a rational map of degree at least $2$. Then $\gamma(f)\le\delta(f)$.  
\end{proposition}

\subsection{Restrictions on $\widehat{N}$ via refined FSI}
\label{subsecn:restrictions}
We now combine the FSI with the possible degeneracies of our normalized quartic Newton maps to obtain some dynamical restrictions on the limits. 
To set up the statement, for a rational map $f$, we denote $\mathrm{Fix_{sa}}(f)$ and $\mathrm{Fix_a}(f)$ the set of superattracting fixed points and the set of attracting fixed points, respectively.  

Now suppose $\{(r_k,s_k)\}\subset\mathrm{nm}_4^\#$ is a sequence converging to $(r,s)\in \Delta \cap \partial  \mathrm{nm}_4^{\#}$. Let $N=H_N\widehat{N}$ be the limit of $N_{\{0,1,r_k,s_k\}}$. Then the possibilities of $(r,s)$ in section \ref{lifting} give us the following restrictions on $\widehat{N}$.
\begin{table}[h!]
\begin{center}
\bgroup
\def\arraystretch{1.5}%
    \begin{tabular}{| l | l | l | l | l | l |}
    \hline
     & $\deg\widehat{N}$ & $\#\mathrm{Fix_{sa}}(\widehat{N})$ & $\#\mathrm{Fix_a}(\widehat{N})$ & $\gamma(\widehat{N})$ & $\delta(\widehat{N})$\\ \hline
    type 1 & 3 & 2 & 1 & 1 or 2 & 1 or 2\\ \hline
    type 2 & 2 & 0 & 2 & 2 & 2\\ \hline
   type 3 & 2 & 1 & 1 & 1 & 1\\
    \hline
    \end{tabular}
    \egroup
\end{center}
\caption{Epstein's invariants for $\widehat{N}$.}
\label{table invariants}
\end{table}

\subsection{Restrictions on limiting cycles via refined FSI} 
\label{subsecn:constraints_via_FSI}
We begin with the following general elementary result.
\begin{lemma}\label{par}
Let $f$ and $f_k$ be analytic on a compact set $K\subset\mathbb{C}$ such that $f_k\to f$ uniformly on $K$. Let $\langle z_k\rangle$ and $\langle w_k\rangle$ be two distinct cycles of $f_k$ in $K$. Let $\Gamma_1$ and $\Gamma_2$ be the limits of $\langle z_k\rangle$ and $\langle w_k\rangle$, respectively. Then $\Gamma_1, \Gamma_2$ are cycles of $f$.  If $\Gamma_1\cap\Gamma_2\not=\emptyset$, then $\Gamma_1=\Gamma_2$ is a parabolic cycle of $f$.   
\end{lemma}
\begin{proof}
Uniform convergence implies $\Gamma_1$ and $\Gamma_2$ are cycles of $f$. So they are same if they have nonempty intersection. Now we show this cycle is parabolic. Let $n_1$ and $n_2$ be the periods of $\langle z_k\rangle$ and $\langle w_k\rangle$, respectively. Assume $z^{(i)}_k$ converges to $z^{(i)}$ and $w^{(i)}_k$ converges to $w^{(i)}$. Suppose $z^{(0)}=w^{(0)}\in\Gamma_1\cap\Gamma_2$. Then there exists analytic functions $h_k(z)$ such that 
$$f_k^{n_1n_2}(z)-z=\prod_{i=0}^{n_1-1}(z-z^{(i)}_k)\prod_{j=0}^{n_2-1}(z-w^{(i)}_k)h_k(z).$$
Hence we have 
$$f^{n_1n_2}(z)-z=(z-z^{(0)})^2\prod_{i=1}^{n_1-1}(z-z^{(i)})\prod_{j=1}^{n_2-1}(z-w^{(i)})h(z)$$
for some holomorphic function $h(z)$. Thus $(f^{n_1n_2})'(z^{(0)})=1$. Note $(f_k^{n_1n_2})'(z_k^{(0)})$ converges to $(f^{n_1n_2})'(z^{(0)})$. Then $(f_k^{n_1})'(z^{(0)}_k)$ converges to $e^{2\pi im/n_2}$ for some $0\le m\le n_2-1$. Hence $(f^{n_1})'(z^{(0)})=e^{2\pi im/n_2}$. Let $p$ be the period of $\Gamma_1$, then $p\mid n_1$. Therefore, $\Gamma_1$ is parabolic. 
\end{proof}
If a rational map $f$ has an $n$-cycle with multiplier $\rho=e^{2\pi p/q}$ and degeneracy $\nu$, then a generic perturbation of $f$ splits this cycle into an $n$-cycle with multiplier close to $\rho$ and a $\nu$-tuple of $nq$-cycles with multipliers close to $1$. In particular,
\begin{lemma}\label{par-att-ind}\cite[Lemma 1]{Epstein00}
Let $f$ be analytic on $U \subset \mathbb{C}$ with a parabolic $n$-cycles $\langle z\rangle$ of multiplier $e^{2\pi ip/q}$. Let $f_k$ be analytic with $f_k\to f$ locally uniformly on $U$, and with $n$-cycle $\langle z_k^{[0]}\rangle$ and $nq$-cycles $\langle z_k^{[1]}\rangle, \dots \langle z_k^{[\nu]}\rangle$ converging to $\langle z\rangle$. If all $\langle z_k^{[j]}\rangle$ are attracting for $k$ sufficiently large then $\langle z\rangle$  is parabolic-attracting or parabolic-indifferent.
\end{lemma}
For degenerating Newton maps with two distinct free (super)attracting cycles, we obtain constraints on limits of cycles via the following. To set up the statement, we place ourselves in the setting of Lemma \ref{lemma:semialg}, and let $\mathcal{H}^\#$ be the lift of a hyperbolic component $\mathcal{H}\subset\mathrm{nm}_4$ of type D. The proposition below says that at least one of the free (super)attracting cycles collides with a root at a hole. 

\begin{proposition}\label{intersection}
Let $\{(r_k,s_k)\}\subset\mathcal{H}^\# \subset \mathrm{nm}_4^\#$ be a degenerating sequence converging to $(r,s)\in \Delta\cap\partial\mathrm{nm}_4^{\#}$. Let $N=H_N\widehat{N}$ be the limit of $N_{\{0,1,r_k,s_k\}}$ in $\overline{\mathrm{Rat}}_4$. Let $\langle z_k\rangle$ and $\langle w_k\rangle$ be the free (super)attracting cycles of $N_{\{0,1,r_k,s_k\}}$. Let $\Gamma_1$ and $\Gamma_2$ be the limits of $\langle z_k\rangle$ and $\langle w_k\rangle$, respectively. Then there exists $i\in\{1,2\}$ such that $\Gamma_{i}\cap\mathrm{Hole}(N)\not=\emptyset$. 

\end{proposition}
\begin{proof}Write $N=H_N\widehat{N}$. Suppose $\Gamma_i\cap\mathrm{Hole}(N)=\emptyset$ for $i=1,2$. Without loss of generality, we may assume $(r,s)$ is of type 1, type 2 or type 3.\par

\noindent{\bf Case 1:} $\mathbf{\Gamma_1\cap\Gamma_2\not=\emptyset}$.  Let $U$ be a small neighborhood of $\Gamma_1\cup\Gamma_2$ such that $\overline{U}\cap\mathrm{Hole}(N)=\emptyset$. Then $N_{\{0,1,r_k,s_k\}}\to\widehat N$ uniformly on $\overline{U}$. By Lemmas \ref{par} and \ref{par-att-ind}, we know $\Gamma_1=\Gamma_2$ is a parabolic-attracting or parabolic-indifferent cycle $\Gamma$ of $\widehat N$. Since $\widehat{N}$ has at least one attracting fixed point, we have 
$$\gamma(\widehat{N})>\gamma_{\Gamma}\ge 2.$$
which contradicts Table \ref{table invariants} of Section \ref{subsecn:restrictions}.\par
\noindent{\bf Case 2:} $\mathbf{\Gamma_1\cap\Gamma_2=\emptyset}$. In this case $\Gamma_1$ and $\Gamma_2$ are both non-repelling and neither of them is a (super)attracting fixed point. 

If both of them are not superattracting, again, the existence of attracting fixed points implies 
$$\gamma(\widehat{N})>\gamma_{\Gamma_1}+\gamma_{\Gamma_2}\ge 2$$
which again contradicts Table \ref{table invariants}.
Now suppose $\Gamma_1$ is superattracting. If $(r,s)$ is of type $1$, we have 
$$\gamma(\widehat{N})\ge\gamma_{
\langle 0\rangle}+\gamma_{\Gamma_2}\ge 1+\gamma_{\Gamma_2}.$$
However, in this case $\delta(\widehat{N})=1$. Thus $\gamma_{\Gamma_2}=0$. If $(r,s)$ is of type $2$ or $3$, then $\gamma_{\Gamma_2}=0$. Hence $\Gamma_2$ is also superattracting, contradicting Table \ref{table invariants}.\par
Thus at least one of the $\Gamma_i$s intersects with $\mathrm{Hole}(N)$. 
\end{proof}

\subsection{Restrictions on limits of critical orbits via analytical estimates}
\label{subsecn:lco}

Suppose $N=H_N\widehat{N}$ is a degenerate Newton map for which two or more roots have collided and $\deg(\widehat{N}) \geq 2$. Then  $N$ has a hole that is an attracting fixed point of $\widehat{N}$.  The following general result about  fixed-points colliding with attracting cycles constrains the orbits of the critical points in a somewhat general setting. The idea is due to Epstein, see \cite[Proposition 5]{Epstein00}.  Our proof below replaces Epstein's analytical argument (specifically, his use of his Lemma 5) with a topological one.

\begin{proposition}\label{Fatou-size}
Let $\{f_k\}$ be a sequence of degree $d\ge 2$ rational maps such that $f_k$ converges to $f=H_f\hat f$. Assume $\deg\hat f\ge 2$ and $0\in\mathrm{Hole}(f)$ is a 
fixed point of $\hat f$. Let $\langle z_k\rangle$ be a (super)attracting cycle of period $n\ge 2$ and let $\Omega^{(\ell)}_k$ be the Fatou component containing $z^{(\ell)}_k$. Suppose that $z_k^{(\ell)}\to z^{(\ell)}$ for $\ell=0,\cdots, \ell-1$ with $z^{(0)}=0$ and $z^{(i)}\not=0$ for some $1\le i\le n-1$. Then 
\begin{enumerate}
\item $\Omega^{(0)}_k$ converges to $0$ in the sense that for each $\epsilon>0$, $\Omega^{(0)}_k \subset \mathbb{D}_\epsilon:=\{|z|<\epsilon\}$ for all $k$ sufficiently large, and 
\item there exists a neighborhood $V$ of $0$ such that $\Omega^{(i)}_k\cap V=\emptyset$ for $k$ sufficiently large.
\end{enumerate}
\end{proposition}
\begin{proof}
Write $n=i+j$. Then $f_k^{j}(z_k^{(i)})=z_k^{(0)}$.  Since $\deg\hat f\ge 2$, $f^\ell$ is well defined for all $\ell\ge 1$, see \cite{DeMarco05}. Pick a small neighborhood $V$ of $0$ such that $\overline{V}$ is away from nonzero holes of $f^n$ and nonzero $z^{(\ell)}$s.  We can further assume that $\overline{V}\cap\cup_{c\in\mathrm{Crit}(\hat f)}\{c,\cdots, \hat f^n(c)\}\subset\{0\}$ and $\overline{V}\cap\cup_{h\in\mathrm{Hole}(f)}\{\hat f(h),\cdots, \hat f^n(h)\}\subset\{0\}$. Now pick $r>0$ such that $0\in\mathbb{D}_r\Subset V$ and $\hat f^i(\mathbb{D}_r)\Subset V$. Let $B$ be the component of $\hat f^{-j}(\mathbb{D}_r)$ containing $z^{(i)}$. Shrinking $\mathbb{D}_r$ if necessary, we can suppose $\overline{V}$ and $B$ are disjoint. Let $B_k$ be the component of $f_k^{-j}(\mathbb{D}_r)$ containing $z_k^{(i)}$. Since $\partial B$ is away from the holes of $f$, from locally uniform convergence off holes, $B_k$ is in a small neighborhood of $B$ for sufficiently large $k$. Hence we have $V$ and $B_k$ are disjoint.\par
Let $W_k^n$ be the component of $f_k^{-i}(B_k)$ containing $z_k^{(0)}$. We claim that $W_k^n\to 0$. Otherwise, since $z_k^{(0)}\to 0$, passing to a subsequence if necessary, there exists $0<s_1<s_2<r$ such that $W_k^n\cap\{s_1\le |z|\le s_2\}\not=\emptyset$ for sufficiently large $k$. Let $z_k\in W_k^n\cap\{s_1\le |z|\le s_2\}$. Then $f_k^i(z_k)\in f_k^i(W_k^n)=B_k$. On the other hand, from locally uniform convergence off holes, $f_k^i(z_k)\in \hat f^i(\mathbb{D}_r)\subset V$. It is impossible.\par
Now let $\mathcal{W}_k^\ell$ be the component of $f_k^{-n\ell}(\mathbb{D}_r)$ containing $z_k^{(0)}$. Then $\mathcal{W}^1_k=W_k^n$ converges to $0$. Hence $\mathcal{W}^1_k\subset\mathbb{D}_r$ for sufficiently large $k$. It follows that $\mathcal{W}^{\ell+1}_k\subset\mathcal{W}^{\ell}_k$. Let $\mathcal{W}_k$ be the interior of the component of $\cap_{\ell\ge 0}\mathcal{W}_k^\ell$ containing $z^{(0)}_k$.  Since $\{f_k^{n\ell}\}$ is bounded, hence normal, on $\mathcal{W}_k$ and $\mathcal{W}_k$ is a connected open set, we have $\mathcal{W}_k\subset\Omega_k^{(0)}$. On the other hand, for any $z\in\Omega_k^{(0)}$, there exists an open neighborhood $U$ of $z$ such that $f_k^{n\ell}(U)\subset\mathbb{D}_r$ for sufficiently large $\ell$. Since $\Omega_k^{(0)}$ is a connected open set, we have $\Omega_k^{(0)}\subset\mathcal{W}_k$. Therefore, we have $\Omega_k^{(0)}=\mathcal{W}_k$. Thus $\Omega^{(0)}_k\to 0$.\par

Note $\Omega_k^{(i)}=f_k^i(\Omega_k^{(0)})\subset f_k^{i}(W^{n}_k)=B_k$ for $k$ sufficiently large. Since $B_k\cap V=\emptyset$. Thus $\Omega_k^{(i)}\cap V=\emptyset$.
\end{proof}

In the above Proposition \ref{Fatou-size}, the hypothesis that the entire cycle does not collide with the hole is necessary.
\begin{example}
Let $N_t=N_{\{-1, -t, 0, t, +1\}}$ and let $f_t=-N_t$. Let $f_0=H_{f_0}\hat f_0$ be the limit of $f_t$. As $t\to 0$,  the superattracting 2-cycle $\{-t, t\}$ converges to the hole $0$ of $f_0$, which is an attracting fixed point of $\hat f_0$. Note $N_t$ is odd and that $f_t\circ f_t=N_t\circ N_t$. Thus $f_t$ and $N_t$ have same Julia sets. Hence, the Fatou components containing $-t$ and $t$ for $f_t$ are unbounded.
\end{example}

\subsection{Classification of hyperbolic components}\label{type}

As in the cases of quadratic rational maps \cite{Milnor93} and of cubic polynomials \cite{Milnor09}, we can classify  hyperbolic components of quartic Newton maps according to the dynamics on the smallest forward-invariant set of basins containing the set of critical points. 
Focusing on the orbits of the additional critical points, we have the following seven types.  That each type arises can be justified by appealing to the classification theory of \cite{Lodge15a, Lodge15b}.  Below we simply give explicit formulas illustrating each case. \par
\textbf{Type A. Adjacent critical points,} with both additional critical points in the same component of the immediate basin of a free (super)attracting cycle.\par
\textbf{Type B. Bitransitive,} with both additional critical points in the immediate basin of a free (super)attracting period cycle, but they do not lie in the same component.\par 
\textbf{Type C. Capture,} with one additional critical point in the immediate basin of a free (super)attracting cycle, the other additional critical point in the basin but not the immediate basin of this cycle.\par
\textbf{Type D. Disjoint (super)attracting orbits,} with both additional critical points in the immediate basins of two distinct free (super)attracting cycles.\par
\textbf{Type IE. Immediate Escape,} with some additional critical point in the immediate basin of a superattracting fixed point.\par
\textbf{Type FE1. One Future Escape,} with one additional critical point in the basin (but not immediate basin) of a superattracting fixed point, while the other additional critical point is in the immediate basin of a free (super)attracting cycle.\par
\textbf{Type FE2. Two Future Escape,} with both additional critical points in the basins (but not immediate basins) of one or two superattracting fixed points.\par
In the remainder of this subsection, we give examples of hyperbolic quartic Newton maps in the components of each type. 

\begin{figure}[h!]
\centering
\begin{minipage}[t]{.5\textwidth}
  \centering
  \includegraphics[width=.9\linewidth]{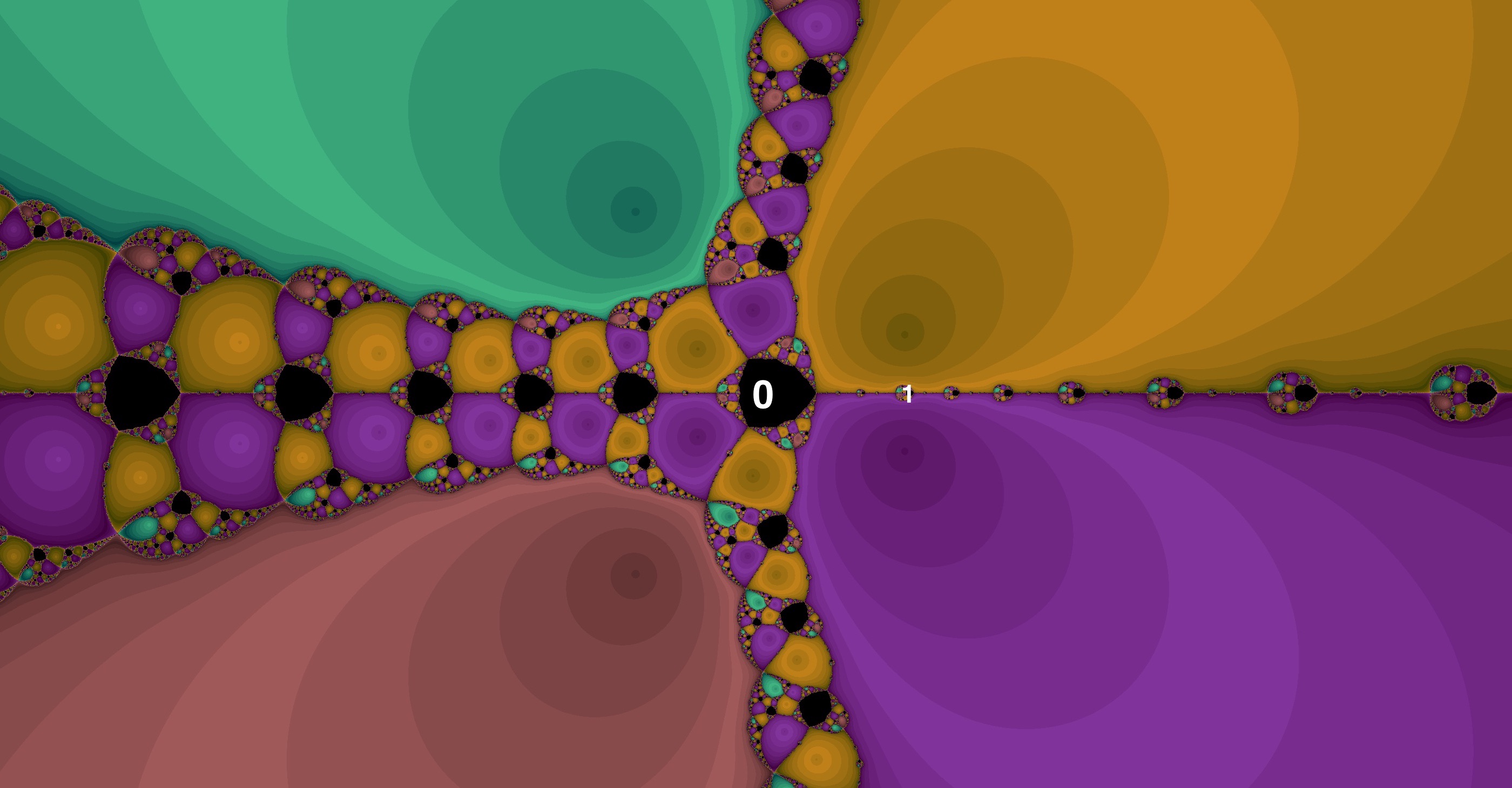}
  \caption{Type A: The Julia set of the Newton map for polynomial $P(z)=z^4/12-z/4+1/4$. The additional critical point $0$ is multiple with the cycle $0\to1\to 0$. }
  \end{minipage}%
\begin{minipage}[t]{.5\textwidth}
  \centering
  \includegraphics[width=.93\linewidth]{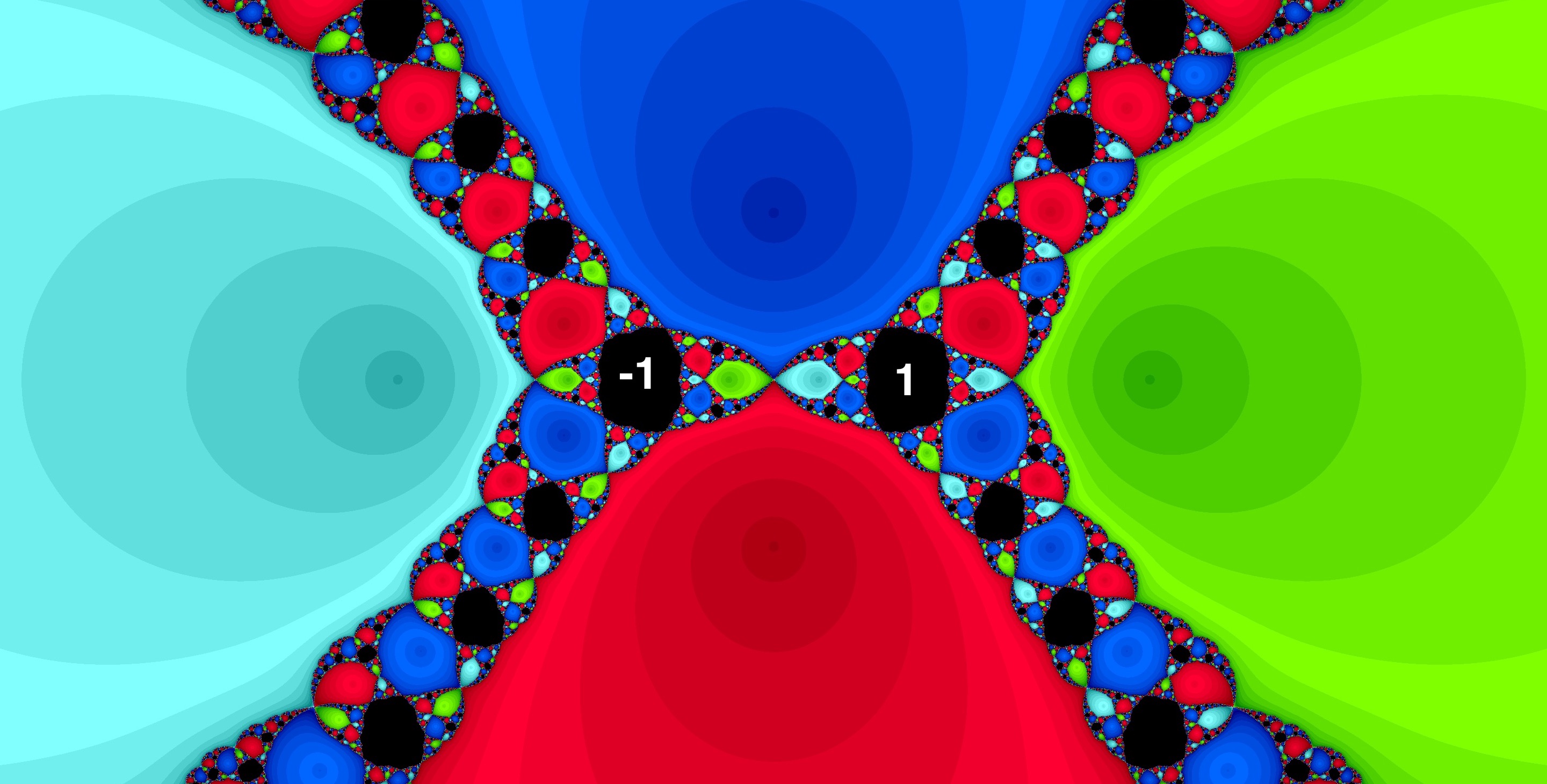}
  \caption{Type B: The Julia set of the Newton map for polynomial $P(z)=z^4/12-z^2/2+z-11/12$. The two additional critical points are $\pm1$ and they are in the two cycle $-1\leftrightarrow 1$.}
\end{minipage}
\end{figure}
\begin{figure}[h!]
\centering
\begin{minipage}[t]{.5\textwidth}
  \centering
  \includegraphics[width=.9\linewidth]{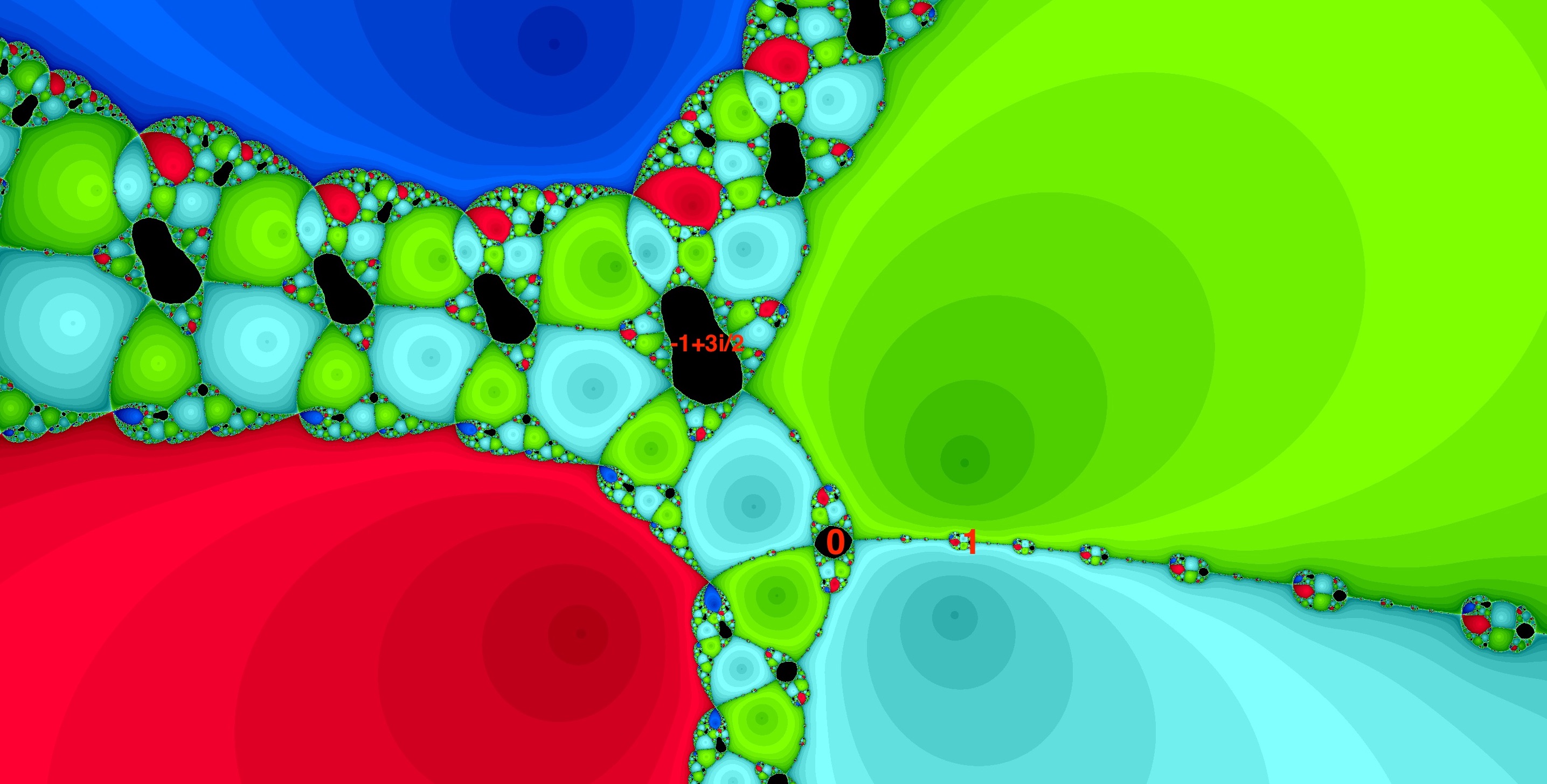}
  \caption{Type C: The Julia set of the Newton map for polynomial $P(z)=z^4/12+(1/6-i/4)z^3-(7/12-i/2)z+7/12-i/2$. The additional critical point $0$ is periodic with orbit $0\to 1\to 0$. The other additional critical point $-1+3i/2$ maps to the immediate basin of $0$.}
  \end{minipage}%
\begin{minipage}[t]{.5\textwidth}
  \centering
  \includegraphics[width=.9\linewidth]{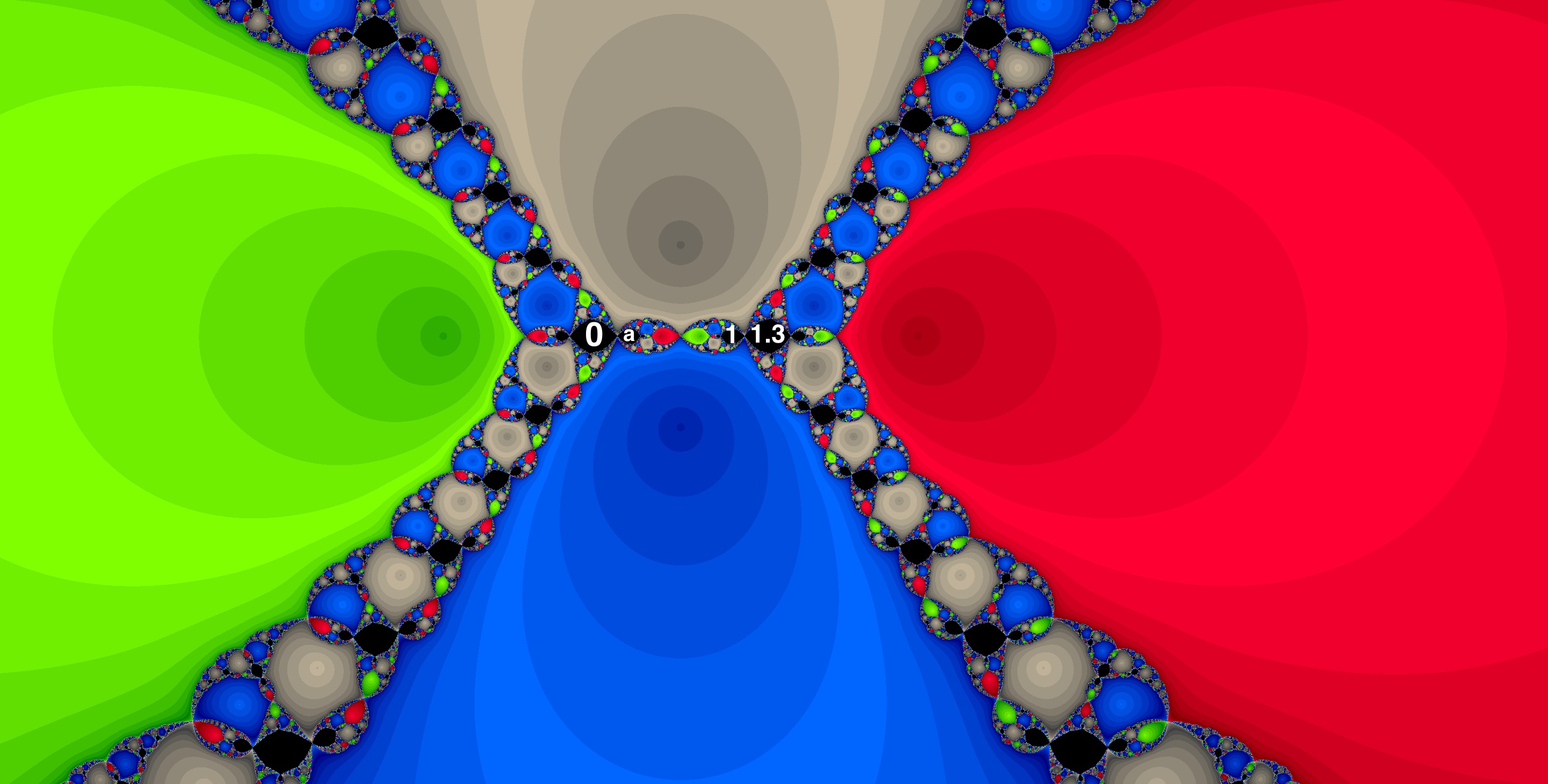}
  \caption{Type D: The Julia set of the Newton map for polynomial $P(z)=z^4/12-13z^3/60+(13/30-1/4)z-13/30+1/4$. The additional critical point $0$ is in the cycle $0\to1\to 0$ and the other additional critical point $1.3$ is in the cycle  $1.3\to a\to 1.3$.}
\end{minipage}
\end{figure}

\begin{figure}[h!]
\centering
\begin{minipage}[t]{.5\textwidth}
  \centering
  \includegraphics[width=.9\linewidth]{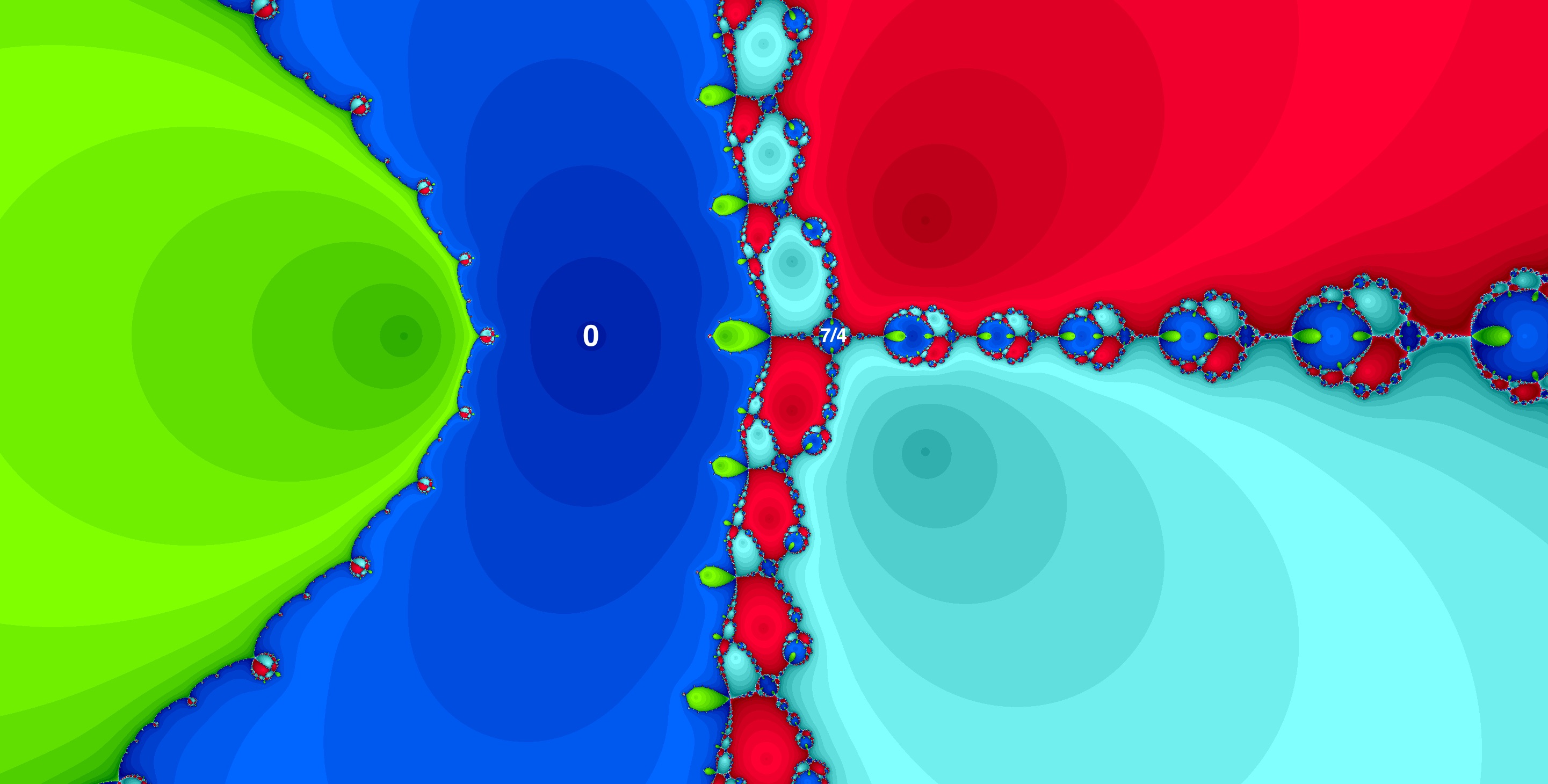}
  \caption{Type IE: The Julia set of the Newton map for polynomial $P(z)=z^4/12-7z^3/24+3z/4$. The additional critical point $0$ is fixed and the other additional critical point $7/4$ is in the basin but not the immediate basin of $0$.}
  \end{minipage}%
\begin{minipage}[t]{.5\textwidth}
  \centering
  \includegraphics[width=.9\linewidth]{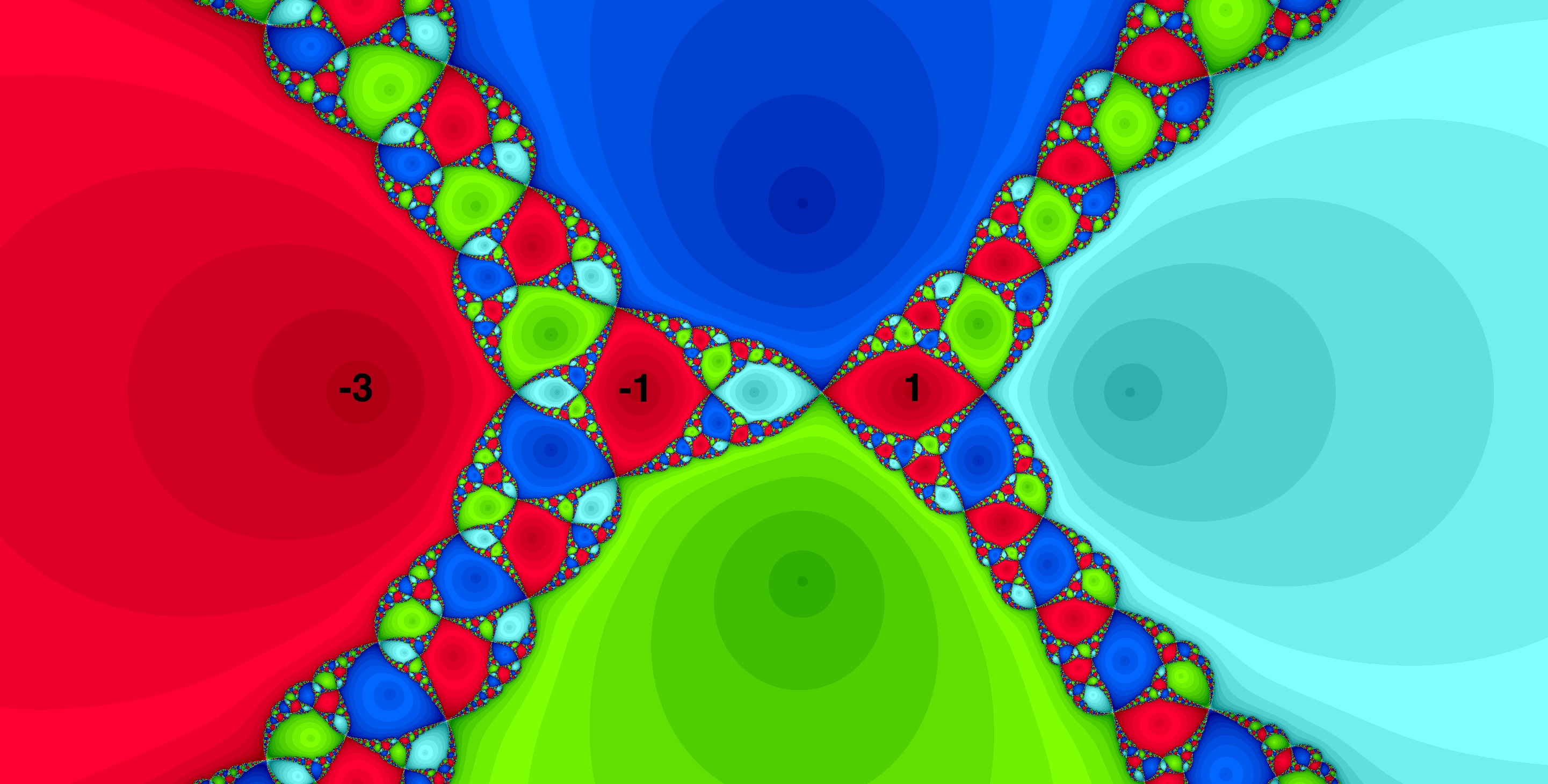}
  \caption{Type FE2: The Julia set of the Newton map for polynomial $P(z)=z^4/4-3z^2/2+z-15/4$. The two additional critical points are $\pm1$ and they maps to the fixed point $-3$ with orbit $-1\to 1\to -3$.}
\end{minipage}
\end{figure}

\section{Berkovich Dynamics of Newton Maps}\label{Berk}
\subsection{Berkovich spaces}
In this subsection, we give brief background of Berkovich space and the related dynamics for rational maps. For more details, we refer \cite{Baker10}.\par
Recall that $\mathbb{L}$ is the completion of the field of formal Puiseux series in variable $t$ over $\mathbb{C}$ with respect to the natural non-Archimedean absolute value. Each element in $\mathbb{L}$ has the form $\sum_{n=0}^\infty a_nt^{q_n}$, where $a_0\in\mathbb{C}\setminus\{0\}$ and $q_n$ increases to $\infty$, and the absolute value is $|\sum_{n=0}^\infty a_nt^{q_n}|=e^{-q_0}$. Let $\mathcal{O}_{\mathbb{L}}$ be the corresponding ring of integers and let $\mathcal{M}_{\mathbb{L}}$ be the unique maximal ideal in $\mathcal{O}_{\mathbb{L}}$. Then the residue field $\mathcal{O}_{\mathbb{L}}/\mathcal{M}_{\mathbb{L}}$ is isomorphic to $\mathbb{C}$.  A rational map $\phi\in\mathbb{L}(z)$ is normalized if the maximum of the absolute values of its coefficients is $1$. For a normalized degree $d\ge0$ rational map $\phi\in\mathbb{L}(z)$, the reduction $\mathrm{red}(\phi)$ is a degree at most $d$ rational map in $\mathbb{C}(z)$ defined by the image of $\phi$ under the map $\mathcal{O}_{\mathbb{L}}(z)\to\mathcal{O}_{\mathbb{L}}/\mathcal{M}_{\mathbb{L}}(z)$. Since every rational map has a unique normalized form, the reduction map extends to $\mathbb{L}(z)$.  If the normalized rational map $\mathbf{f}\in\mathbb{L}(z)$ is induced by a holomorphic family $\{f_t\}$, then the reduction $\mathrm{red}(\mathbf{f})=\lim\limits_{t\to 0}f_t$.\par
The Berkovich space $\mathbf{P}^1$ over $\mathbb{L}$ is a uniquely arcwise connected compact Hausdorff topological space which therefore has a tree structure. It contains 4 types of points. The projective space $\mathbb{P}^1_\mathbb{L}$ is contained in $\mathbf{P}^1$ as type I points. The type II points and type III points correspond to closed disks in $\mathbb{L}$ depending on whether the radius is in the value group $|\mathbb{L}^\times|$ or not. The type IV points are related to a decreasing sequence of closed disks in $\mathbb{L}$ with empty intersection. The type II point corresponding to the closed unit disk in $\mathbb{L}$ is called the \textit{Gauss point} and denoted by $\xi_g$. For $\xi\in\mathbf{P}^1$, the set of the connected components of $\mathbf{P}^1\setminus\{\xi\}$ is called the \textit{tangent space} at $\xi$ and denoted by $T_{\xi}\mathbf{P}^1$. At each type II point $\xi$, the tangent space $T_\xi\mathbf{P}^1$ can be identified to $\mathbb{P}^1$. In particular, at the Gauss point $\xi_g$, this identification is canonical. For each $\vec{v}\in T_\xi\mathbf{P}^1$, the \textit{Berkovich ball} $\mathbf{B}_{\xi}(\vec{v})^-$ is the corresponding connected component of $\mathbf{P}^1\setminus\{\xi\}$. We call the point $\xi$ to be the \textit{boundary point} of $\mathbf{B}_{\xi}(\vec{v})^-$. For $\vec{v}\in T_{\xi_g}\mathbf{P}^1$ and two type I points $\xi_1, \xi_2\in\mathbf{B}_{\xi_g}(\vec{v})^-$, denote by $\xi_1\vee\xi_2$ the unique intersection of segments $[\xi_1,\xi_g]$ and $[\xi_2,\xi_g]$. For a nontrivial closed finite subtree $T\subset\mathbf{P}^1$ and a point $\xi\in\mathbf{P}^1\setminus T$, the \textit{projection} $\pi_{T}(\xi)$ from $\xi$ to $T$ is the point in $T$ at which the points $\xi$ and $\xi'$ are in different directions for any point $\xi'\in T\setminus\{\pi_{T}(\xi)\}$. If $\xi\in T$, we set $\pi_{T}(\xi)=\xi$.\par
Each rational map $\phi\in\mathbb{L}(z)$ extends to a rational map on $\mathbf{P}^1$. To abuse notation, we also write $\phi$ for the extension. At each point $\xi\in\mathbf{P}^1$, the map $\phi$ induces a map $T_\xi\phi:T_\xi\mathbf{P}^1\to T_{\phi(\xi)}\mathbf{P}^1$. If $\phi(\xi_g)=\xi_g$, then $T_{\xi_g}\phi=\mathrm{red}(\phi)$ is nonconstant. A point $\xi\in\mathbf{P}^1$ belongs to the \textit{Julia set} $J_{\mathrm{Ber}}(\phi)$ if for all neighborhoods $U$ of $\xi$, we have that $\cup\phi^n(U)$ omits at most two points. The complement of the Julia set is the \textit{Fatou set} $F_{\mathrm{Ber}}(\phi)$. As in complex settings, we can define attracting, indifferent and repelling cycles for $\phi$ in $\mathbf{P}^1$. If $U$ is a $p$-periodic Fatou component of $\phi$, then $U$ is either the immediate basin of an attracting $p$-periodic point, which is the Fatou component containing this attracting periodic point, or a Rivera domain, which is a Fatou component on which $\phi^p$ is a bijection.\par
If the rational map $\mathbf{f}\in\mathbb{L}(z)$ is induced by a holomorphic family $\{f_t\}$, then a $p$-periodic point $z(t)$ of $f_t$, which is algebraic in $t$ and therefore can be regarded as an element in $\mathbb{L}$, induces a $p$-periodic point $\mathbf{z}$ of  $\mathbf{f}$. Recall that a direction $\vec{v}\in T_{\xi}\mathbf{P}^1$ is a \textit{bad direction} of $\mathbf{f}$ if $\mathbf{f}(\mathbf{B}_{\xi}(\vec{v}))=\mathbf{P}^1$.  The bad directions of $\mathbf{f}$ are related to the holes of $f_0$, see \cite[Lemma 3.17]{Faber13I}. To end this subsection, we formulate Proposition \ref{Fatou-size} in Berkovich dynamics view.
\begin{proposition}
Let $\{f_t\}$ be a holomorphic family of rational maps and let $\mathbf{f}\in\mathbb{L}(z)$ be the induced map. Suppose that $\deg\mathrm{red}(\mathbf{f})\ge 2$ and $\vec{v}\in T_{\xi_g}\mathbf{P}^1$ is a bad direction of $\mathbf{f}$. Let $\langle z(t)\rangle$ be a (super)attracting $n$-cycle of $f_t$ such that $\mathbf{z}^{(0)}\in\mathbf{B}_{\xi_g}(\vec{v})^-$ and $\mathbf{z}^{(i)}\not\in\mathbf{B}_{\xi_g}(\vec{v})^-$ for some $1\le i\le n-1$. If $T_{\xi_g}\mathbf{f}(\vec{v})=\vec{v}$, then there exists a critical point $\mathbf{c}$ of $\mathbf{f}$ such that $\mathbf{f}^\ell(\mathbf{c})\in\mathbf{B}_{\xi_g}(\vec{v})^-$ and $\mathbf{f}^s(\mathbf{c})\not\in\mathbf{B}_{\xi_g}(\vec{v})^-$ for some $0\le \ell, s\le n-1$. In particular, let $\vec{w}\in T_{\xi_g}\mathbf{P}^1$ be the direction such that $\mathbf{c}\in\mathbf{B}_{\xi_g}(\vec{w})^-$. If $\vec{w}$ is not a bad direction, then $\mathrm{red}(\mathbf{c})$ is prefixed under iteration of $\mathrm{red}(\mathbf{f})$.
\end{proposition}

\subsection{Berkovich dynamics of quartic Newton maps}
We now consider the Berkovich dynamics induced by degenerating families corresponding to the three cases in section \ref{lifting}.  For more Berkovich dynamics of general Newton maps, we refer \cite{Nie-thesis}.\par
Suppose $\mathbf{r}, \mathbf{s}\in\mathbb{L}\setminus\{0,1\}$, and define $\mathbf{N}(z):=\mathbf{N}_{\{0, 1, \mathbf{r},\mathbf{s}\}}(z)$ be the Newton map for the polynomial $\mathbf{P}(z)=z(z-1)(z-\mathbf{r})(z-\mathbf{s})\in\mathbb{L}(z)$. 

Refining the classification of degenerations from section \ref{lifting}, we say the map $\mathbf{N}$ is of
\begin{itemize}
\item  \textit{type 1} if $\mathrm{red}(\mathbf{r})=0$ and $\mathrm{red}(\mathbf{s})\in\mathbb{C}\setminus\{0,1\}$;
\item  \textit{type 2} if $\mathrm{red}(\mathbf{r})=0$ and $\mathrm{red}(\mathbf{s})=1$;
\item  \textit{type 3a} if $\mathrm{red}(\mathbf{r})=\mathrm{red}(\mathbf{s})=0$ and $|\mathbf{r}|<|\mathbf{s}|$;
\item  \textit{type 3b} if $\mathrm{red}(\mathbf{r})=\mathrm{red}(\mathbf{s})=0$ and $|\mathbf{r}-\mathbf{s}|=|\mathbf{r}|=|\mathbf{s}|$.
\end{itemize}
In type 3a, the root corresponding to $r$ converges to zero at a faster rate in $t$ than does $s$, while in type 3b they converge at asymptotically the same rate in $t$. \par
Viewing types 1, 2, 3a, 3b as normal forms, we now study their structure as Berkovich dynamical systems.  We follow \cite{Nie-thesis}, specializing to the quartic case.
Define 
$$H_{\mathrm{fix}}:=\mathrm{Hull}(\{0, 1, \mathbf{r}, \mathbf{s},\infty\}).$$
Then $\mathbf{N}(H_{\mathrm{fix}})=H_{\mathrm{fix}}$ \cite[Lemmas 3.13, 3.19]{Nie-thesis}. Set 
$$V_{\mathrm{rep}}:=\{\xi\in H_{\mathrm{fix}}:\mathrm{Val}_{H_{\mathrm{fix}}}(\xi)\ge 3\},$$
where $\mathrm{Val}_{H_{\mathrm{fix}}}(\xi)$ is the valence of the point $\xi$ in the tree $H_{\mathrm{fix}}$.
Then $V_{\mathrm{rep}}$ is the set of type II repelling fixed points of $\mathbf{N}$ \cite[Lemma 3.13]{Nie-thesis}. We illustrate the corresponding $H_{\mathrm{fix}}$ and  $V_{\mathrm{rep}}$ for these four types Newton maps in Figure \ref{four-types}.
\begin{figure}[h!]
\begin{subfigure}{.4\textwidth}
  \centering
  \includegraphics[width=.7\linewidth]{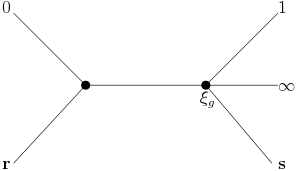}
  \caption{Type 1}
  %\label{fig:sfig1}
\end{subfigure}
\begin{subfigure}{.4\textwidth}
  \centering
  \includegraphics[width=.7\linewidth]{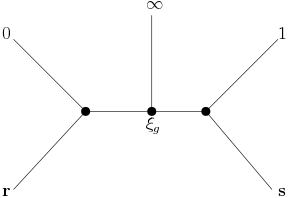}
  \caption{Type 2}
 % \label{fig:sfig2}
\end{subfigure}
\begin{subfigure}{.4\textwidth}
  \centering
  \includegraphics[width=.7\linewidth]{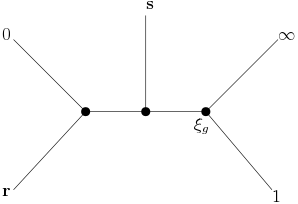}
  \caption{Type 3a}
  %\label{fig:sfig1}
\end{subfigure}
\begin{subfigure}{.4\textwidth}
  \centering
  \includegraphics[width=.7\linewidth]{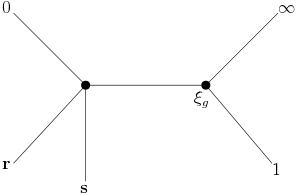}
  \caption{Type 3b}
 % \label{fig:sfig2}
\end{subfigure}
\caption{The convex hulls $H_{\mathrm{fix}}$ for different types of $\mathbf{N}$. The set $V_{\mathrm{rep}}$ consists of the black dots.}
\label{four-types}
\end{figure}\par

Using this structure, we now show that types 1, 2, 3a and 3b cover all possibilities for degenerations in moduli space:

\begin{lemma}\label{types}
Let $\{N_t\}$ be a holomorphic family of marked quartic Newton maps. Suppose that $[N_t]\to\infty$ in $\mathrm{nm}_4$. Then there exists an affine map $M_t(z)\in \mathbb{L}(z)$ such that the induced map of $M_t^{-1}\circ N_t\circ M_t$ is of type $1, 2, 3a$ or $3b$.
\end{lemma}
\begin{proof}
For $1\le i\le 4$, let $r_i(t)$ be the surperattracting fixed points of $N_t$. Regard $r_i(t)$ as a point in $\mathbb{L}$ and denote by $\mathbf{r}_i$. Now consider the convex hull  $H=\mathrm{Hull}(\{\mathbf{r}_1, \cdots, \mathbf{r}_4,\infty\})$. Then $H$ is a subtree and there are at most three points whose valences in $H$ are at least three. If there is only one such point, then $[N_t]\not\to\infty$. If there are two such points, we can normalize $N_t$ such that the induced map $\mathbf{N}$ is of type $1$ or $3b$. If there are three such points, we normalize $N_t$ such that the induced map $\mathbf{N}$ is of type $2$ or $3a$. 
\end{proof}
In the remainder of this subsection, we assume $N_t=N_{(0,1,r(t),s(t))}$ is a holomorphic family with induced map $\mathbf{N}$ of type 1, 2, 3a, or 3b. Write $N_0=H_N\widehat N$. Then $N_t \to \widehat{N}$ locally uniformly outside the set of at most two holes, which are attracting fixed-points of $\widehat{N}$. Moreover, we have $\mathrm{red}(\mathbf{N})=\widehat N$.\par

Continuing our analysis of the dynamical structure, define
$$H_{\mathrm{rep}}:=\mathrm{Hull}(V_{\mathrm{rep}}).$$
Then any point $\xi\in H_{\mathrm{rep}}$ is a fixed point of $\mathbf{N}$. In fact, the set of non type I fixed points of $\mathbf{N}$ is $H_{\mathrm{rep}}\cup(\xi_g,\infty)$ \cite[Lemma 3.13]{Nie-thesis}. Note for any point $\xi\in\mathbf{P}^1$ and any $\vec{v}\in T_{\xi}\mathbf{P}^1$, if the Berkovich ball $\mathbf{B}_{\xi}(\vec{v})^-$ is disjoint with $V_{\mathrm{rep}}$, then $\mathbf{N}(\mathbf{B}_{\xi}(\vec{v})^-)\not=\mathbf{P}^1$ \cite[Corollary 3.17]{Nie-thesis}. 
For the dynamics on these trees, we have 
\begin{lemma}\cite[Lemma 3.19]{Nie-thesis}\label{attracting-root}
For any $\xi\in\{0,1,\mathbf{r},\mathbf{s}\}$, the segment $(\xi, \pi_{H_{\mathrm{rep}}}(\xi))$ is forward invariant, and any point on the segment $(\xi, \pi_{H_{\mathrm{rep}}}(\xi))$ is attracted to $\xi$ .
\end{lemma}
Recall that an \textit{annulus} in $\mathbf{P}^1$ is the intersection of two Berkovich balls $\mathbf{B}_1, \mathbf{B}_2$ with distinct boundary points such that $\mathbf{B}_1\cup \mathbf{B}_2=\mathbf{P}^1$.
\begin{lemma}\cite[Lemma 3.18]{Nie-thesis}\label{fixed-Fatou}
Let $U$ be a component of $F_{\mathrm{Ber}}(\mathbf{N})$ fixed by $\mathbf{N}$. 
\begin{enumerate}
\item $U$ is a fixed Rivera domain if and only if $U$ is a component of $\mathbf{P}^1\setminus V_{\mathrm{rep}}$ which is either an annulus or a ball containing $\infty$.
\item $U$ is a fixed attracting domain if and only if $U$ is a component of $\mathbf{P}^1\setminus V_{\mathrm{rep}}$ which contains some $\xi\in\{0,1,\mathbf{r},\mathbf{s}\}$.
\end{enumerate}
\end{lemma}

The following result gives us basic information about the locations of the Berkovich type I cycles encoding asymptotics of cycles in the dynamical complex plane. 

\begin{proposition}\label{cycles}
Let $\langle\mathbf{z}\rangle$ be a type I periodic cycle of period $n\ge 2$. Then for $0\le i\le n-1$,
$$\pi_{H_{\mathrm{fix}}}(\mathbf{z}^{(i)})\in V_{\mathrm{rep}}.$$
\end{proposition}
\begin{proof}
We first claim that any periodic point in a fixed Rivera domain is fixed.  For any type II point $\xi\in(H_{\mathrm{rep}}\setminus V)\cup(\xi_g,\infty)$, let $\mathbf{M}\in\mathbb{L}(z)$ be an affine map such that $\mathbf{M}(\xi_g)=\xi$. It is easy to check $\mathrm{red}(\mathbf{M}^{-1}\circ\mathbf{N}\circ\mathbf{M})$ is a degree $1$ rational map with no nonfixed periodic points.  Thus any point in the fixed Rivera domains can not be a periodic point of period at least $2$.\par
By Lemma \ref{fixed-Fatou}, we only need to show $\pi_{H_{\mathrm{fix}}}(\mathbf{z}^{(i)})\not\in(\xi, \pi_{H_{\mathrm{rep}}}(\xi))$ for all $\xi\in\{0,1,\mathbf{r},\mathbf{s}\}$. It immediately follows from Lemma \ref{attracting-root}.
\end{proof}
For any critical point $\mathbf{c}\in\mathrm{Crit}(\mathbf{N})$, the projection $\pi_{H_{\mathrm{rep}}}(\mathbf{c})\in V_{\mathrm{rep}}$ since the points in $H_{\mathrm{rep}}\setminus V_{\mathrm{rep}}$ have local degree $1$. For any point $\xi\in V_{\mathrm{rep}}$, let $\mathbf{M}\in\mathbb{L}(z)$ be an affine map such that $\mathbf{M}(\xi_g)=\xi$. Then the reduction of  $\mathbf{M}^{-1}\circ\mathbf{N}\circ\mathbf{M}$ is nonconstant. Moreover, $\deg_\xi(\mathbf{N})=\deg\mathrm{red}(\mathbf{M}^{-1}\circ\mathbf{N}\circ\mathbf{M})$. Thus, for $\xi\in H_{\mathrm{rep}}$, there are $2\deg_\xi(\mathbf{N})-2$ critical points $\mathbf{c}\in\mathrm{Crit}(\mathbf{N})$ such that $\pi_{H_{\mathrm{rep}}}(\mathbf{c})=\xi$. As in complex setting, we say a critical point $\mathbf{c}\in\mathbb{P}^1_{\mathbb{L}}$ of Newton map $\mathbf{N}$ is \textit{additional} if $\mathbf{P}''(\mathbf{c})=0$ and it is \textit{free} if $\mathbf{N}(\mathbf{c})\not=\mathbf{c}$.
\begin{proposition}\cite[Lemma 3.16]{Nie-thesis}\label{critical-projection}
Let $\mathbf{c}_1$ and $\mathbf{c}_2$ be the additional critical points. Denote $\Sigma$ the set of projections $\pi_{H_{\mathrm{rep}}}(\mathbf{c}_1)$ and  $\pi_{H_{\mathrm{rep}}}(\mathbf{c}_2)$. Then 
\begin{enumerate}
\item If $\mathbf{N}$ is of type $1$ or $2$, $\Sigma=\{\xi_g\}$. 
\item If $\mathbf{N}$ is of type $3a$ or $3b$, $\Sigma=\{\xi_g,0\vee\mathbf{s}\}$
\end{enumerate}
\end{proposition}
Now denote by $\mathrm{FreeCrit}(\mathbf{N})$ the set of the free critical points of the map $\mathbf{N}$.  The reductions of the free critical points of $\mathbf{N}$ give critical points $c_i$ of $\widehat{N}$. Proposition \ref{critical-projection} implies that their orbits under $\widehat{N}$ are constrained according to the three cases. Table \ref{table free} lists the number of critical points $c_i$ of $\mathbf{N}$ coming from the reduction of free critical points, the locations of the attracting fixed points of $\widehat{N}$, and the behavior of the orbit of $c_i$ under iteration of $\widehat{N}$. The precise statement is in the following Proposition \ref{Berk-dyn}. 
\begin{table}[h!]
\begin{center}
\bgroup
\def\arraystretch{1.5}%
    \begin{tabular}{| l | l | l | l | }
    \hline
     & $\#\mathrm{FreeCrit}(\mathbf{N})$& $\mathrm{Fix_{a}}(\widehat{N})$ & orbit of $c_i$'s under $\widehat{N}$ \\ \hline
    type 1 & 1 or 2 & \{0\}& for some $c_i$, $0 \neq c_i \to 0$ with infinite orbit\\ \hline
    type 2 & 2 & \{0, 1\}& $0 \neq c_i \to i, \; i=0,1$, each with infinite orbit \\ \hline
   type 3a & 2 & \{0\} & $c_0=0, 0 \neq c_1 \to 0$ with infinite orbit \\ \hline
   type 3b & 1 or 2 & \{0\} & $c_0=0, 0 \neq c_1 \to 0$ with infinite orbit \\
    \hline
    \end{tabular}
    \egroup
\end{center}
\caption{The number of free critical points of $\mathbf{N}$, locations of the attracting fixed points of $\widehat{N}$, and orbit of reductions $c_i$ under $\widehat{N}$.}
\label{table free}
\end{table}\par

\begin{proposition}\label{Berk-dyn}
There is a free critical point $\mathbf{c}$ such that the projection $\pi_{H_{\mathrm{fix}}}(\mathbf{c})=\xi_g$. In particular,
\begin{enumerate}
\item If $\mathbf{N}$ is of type $1$, the reduction of a free critical point is attracted to $0$ under $\widehat{N}$. 
\item If $\mathbf{N}$ is of type $2$, the reduction of a free critical point is attracted to $0$ under $\widehat{N}$ and the reduction of the other free critical point is attracted to $1$ under $\widehat{N}$.
\item If $\mathbf{N}$ is of type $3a$ or $3b$, the reduction of a free critical point is attracted to $0$ under $\widehat{N}$, and the reduction of the other additional critical point is $0$.
\end{enumerate}
\end{proposition}

\section{Constraints on limiting cycles via Berkovich dynamics}\label{cycle}
In this section, we apply the study of Berkovich dynamics from section \ref{Berk} to obtain finer results on limits $\Gamma_1, \Gamma_2$ of the free attracting cycles for holomorphic family degenerating along a sequence in a hyperbolic component of type D.  

To this end, let $\{N_t\}$ be a holomorphic family of Newton maps for polynomials $z(z-1)(z-r(t))(z-s(t))$ converging to $N=H_N\widehat{N}$ with induced map $\mathbf{N}$ of type 1, 2, 3a, or 3b. We further assume a sequence $N_{t_k}$ lies in a lift $\mathcal{H}^\#$ of a type D hyperbolic component, and again denote the two free (super)attracting cycles, by  $\langle z_k\rangle$ and $\langle w_k\rangle$. Let $\Gamma_1$ and $\Gamma_2$ be the limits of $\langle z_k\rangle$ and $\langle w_k\rangle$, respectively. Proposition \ref{intersection} asserts that there is $i_0\in\{1,2\}$ such that $\Gamma_{i_0}\cap\mathrm{Hole}(N)\not=\emptyset$. Without loss of generality, we let $i_0=1$. The following result gives more precise information about the intersection of limiting cycles and $\mathrm{Hole}(N)$. 

\begin{proposition}\label{cycle123}
Let $\Gamma_1$ and $\Gamma_2$ be as above.
\begin{enumerate}
\item If $\mathbf{N}$ is of type $1$, then $\{0\}\subsetneqq\Gamma_1$.
\item If $\mathbf{N}$ is of type $2$, then $\{a\}\subsetneqq\Gamma_1$ for some $a\in\{0,1\}$.
\item If $\mathbf{N}$ is of type $3a$, then $\{0\}\subset\Gamma_1$. Moreover, if $\Gamma_1=\{0\}$,  let $\{M_t(z)\}$ be a family of scaling maps such that the induced map $\mathbf{M}(\xi_g)=0\vee\mathbf{s}$. Set $\tilde z_k=M_{t_k}^{-1}(z_k)$ and let $\widetilde\Gamma_1$ be the limit of $\langle\tilde z_k\rangle$. Then $\{0\}\subsetneqq\widetilde\Gamma_1$.
\item If $\mathbf{N}$ is of type $3b$, then $\{0\}\subset\Gamma_1$ and $\{0\}\subset\Gamma_2$. Moreover, if $\Gamma_1=\{0\}$, then $\{0\}\subsetneqq\Gamma_2$.
\end{enumerate}
\end{proposition}
\begin{proof}
We treat each type in turn.

If $\mathbf{N}$ is of type $1$, then $\mathrm{Hole}(N)=\{0\}$. If $\Gamma_1=\{0\}$, then by Proposition \ref{cycles}, we have $|z_k^{(i)}|=O(|r(t_k)|)$ for all $0\le i\le n-1$. Let $\{M_t(z)\}$ be a family of affine maps such that the induced map $\mathbf{M}(\xi_g)=0\vee\mathbf{r}$. Then $\langle M^{-1}_{t_k}(z_k)\rangle$ converges to a nonrepelling cycle of $\mathrm{red}(\mathbf{M}^{-1}\circ\mathbf{N}\circ\mathbf{M})$. Note $\mathrm{red}(\mathbf{M}^{-1}\circ\mathbf{N}\circ\mathbf{M})$ is a quadratic rational map with two fixed critical points. Thus it has no other nonrepelling cycles. Thus $\{0\}\subsetneqq\Gamma_1$.\par
If $\mathbf{N}$ is of type $2$, then $\mathrm{Hole}(N)=\{0,1\}$. Since $\Gamma_1\cap\mathrm{Hole}(N)\not=\emptyset$, without of lose generality, we suppose $0\in\Gamma_1$. Then it follows from an analogue of the case when $\mathbf{N}$ is of type $1$.\par
If $\mathbf{N}$ is of type $3a$, then $\mathrm{Hole}(N)=\{0\}$. If $0\not\in\Gamma_1$, then $\Gamma_1$ is a nonrepelling cycle of $\widehat N$. It is impossible since $\widehat{N}$ is a quadratic rational maps with two (super)attracting fixed points. If $\Gamma_1=\{0\}$, then $\widetilde\Gamma_1$ is either a nonrepelling cycle of $\mathrm{red}(\mathbf{M}^{-1}\circ\mathbf{N}\circ\mathbf{M})$ or $0\in\widetilde\Gamma_1$. Note $\mathrm{red}(\mathbf{M}^{-1}\circ\mathbf{N}\circ\mathbf{M})$ is a quadratic rational map with two (super)attracting fixed points. Hence it has no other nonrepelling cycles. Thus $0\in\widetilde\Gamma_1$. If $\widetilde\Gamma_1=\{0\}$, then by Proposition \ref{cycles}, we have $|\tilde z_k^{(i)}|=O(|r(t_k)|)$. Let $\{L_t(z)\}$ be a family of affine maps such that the induced map $\mathbf{L}(\xi_g)=0\vee\mathbf{r}$. Then $\langle L^{-1}_{t_k}(z_k)\rangle$ converges to a nonrepelling cycles of the map $\mathrm{red}(\mathbf{L}^{-1}\circ\mathbf{N}\circ\mathbf{L})$. Note the $\mathrm{red}(\mathbf{L}^{-1}\circ\mathbf{N}\circ\mathbf{L})$ is a quadratic rational map with two superattracting fixed points. Thus it has no other nonrepelling cycles. Thus $\{0\}\subsetneqq\widetilde\Gamma_1$.\par
If $\mathbf{N}$ is of type $3b$, we have $\mathrm{Hole}(N)=\{0\}$. Since $\widehat{N}$ is a quadratic rational maps with two (super)attracting fixed points, $0\not\in\Gamma_1\cap\Gamma_2$. If $\Gamma_1=\Gamma_2=\{0\}$, then again by Proposition \ref{cycles}, we have $|z_k^{(i)}|=O(|r(t_k)|)$. Let $\{L_t(z)\}$ be a family of affine maps such that the induced map $\mathbf{L}(\xi_g)=0\vee\mathbf{r}$. Let $\widetilde\Gamma_1$ and $\widetilde\Gamma_2$ be the limits of $\langle L^{-1}_{t_k}(z_k)\rangle$ and $\langle L^{-1}_{t_k}(w_k)\rangle$, respectively. Let $\overline{N}$ be the reduction of $\mathbf{L}^{-1}\circ\mathbf{N}\circ\mathbf{L}$. Then $\overline{N}$ is a cubic rational map with three superattracting fixed points. The FSI implies that $\widetilde\Gamma_1\cap\widetilde\Gamma_2\not=\emptyset$. Then by Lemmas \ref{par} and \ref{par-att-ind}, we know $\widetilde\Gamma_1=\widetilde\Gamma_2$ is either parabolic-attracting or parabolic-indifferent. Then the Epstein's invariant $\gamma(\overline{N})\ge 2$. However, the number of infinite tails of critical orbits is $\delta(\overline{N})\le 1$. It contradicts to Proposition \ref{number}.
\end{proof}

\section{Boundedness}\label{compact}
In this section, we prove the Main Lemma. 

\begin{proof}[Proof of the Main Lemma]
Suppose $\mathcal{H}$ is a type $D$ component and $N_t$ a holomorphic family as in the statement.  Suppose to the contrary that $[N_{t_k}]\to\infty$ in $\mathrm{nm}_4$ for some sequence $t_k \to 0$. Then $N_t\to N=H_N\widehat N\in\overline{\mathrm{Rat}}_d = \mathbb{P}^{2d+1}$. 

By Lemma \ref{types}, we may and do assume that the induced map $\mathbf{N}$ is of type $1, 2, 3a$ or $3b$. Let $\langle z_k\rangle$ and $\langle w_k\rangle$ be the free attracting cycles of $N_{t_k}$. Suppose that they are of periods $n_1\ge 2$ and $n_2\ge 2$, respectively. Let $c_1(t)$ and $c_2(t)$ be the two additional critical points of $N_t$. Assume $c_1(t_k)$ is in the immediate basin of $\langle z_k\rangle$ and $c_2(t_k)$ is in the immediate basin of $\langle w_k\rangle$. Let $\Gamma_1$ and $\Gamma_2$ be the limits of $\langle z_k\rangle$ and $\langle w_k\rangle$, respectively. By Proposition \ref{intersection}, we may assume $\Gamma_1\cap\mathrm{Hole}(N)\not=\emptyset$.\par

Again we treat each type in turn. \par

If $\mathbf{N}$ is of type $1$, the unique hole $0$ of $N$ is an attracting fixed point of $\widehat N$. Then by Propositions \ref{Fatou-size} and \ref{cycle123}, there exists $0\le q_1\le n_1-1$ such that $N_{t_k}^{q_1}(c_1(t_k))$ converges to $0$.  It follows that for the induced map $\mathbf{N}$, the reduction $\mathrm{red}(\mathbf{N}^{q_1}(\mathbf{c}_1))=0$. Thus the critical point $\mathrm{red}(\mathbf{c}_1)$ of $\widehat N$ satisfies $\widehat N^{q_1}(\mathrm{red}(\mathbf{c}_1))=0$. Note there is a nonzero critical point of $\widehat N$ in the immediate basin of $0$. Thus $0\in\Gamma_2$. By Proposition \ref{cycle123}, we have $\{0\}\subsetneqq\Gamma_2$. Then by Proposition \ref{Fatou-size}, there exists $0\le q_2\le n_2-1$ such that $N_{t_k}^{q_2}(c_2(t_k))$ converges to $0$. It follows that for the induced map $\mathbf{N}$, the reduction $\mathrm{red}(\mathbf{N}^{q_2}(\mathbf{c}_2))=0$. It contradicts to Proposition \ref{Berk-dyn}.\par
If $\mathbf{N}$ is of type $2$, by Propositions  \ref{Fatou-size} and \ref{cycle123}, there exists $0\le q_1\le n_1-1$ such that $N_{t_k}^{q_1}(c_1(t_k))$ converges to a point in $\mathrm{Hole}(N)$. It follows that for the induced map $\mathbf{N}$, the reduction $\mathrm{red}(\mathbf{N}^{q_1}(\mathbf{c}_1)$ is either $0$ or $1$. It contradicts to Proposition \ref{Berk-dyn}.\par
If $\mathbf{N}$ is of type $3a$, by Propositions \ref{cycle123}, we know $0\in\Gamma_1\cap\Gamma_2$. If $\{0\}\subsetneqq\Gamma_1$ and  $\{0\}\subsetneqq\Gamma_2$, by Proposition \ref{Fatou-size}, there exist $0\le q_1\le n_1-1$ and $0\le q_2\le n_2-1$  such that $N_{t_k}^{q_1}(c_1(t_k))$ converges to $0$ and $N_{t_k}^{q_2}(c_2(t_k))$ converges to $0$.  It follows that for the induced map $\mathbf{N}$, the reductions $\mathrm{red}(\mathbf{N}^{q_1}(\mathbf{c}_1))=0$ and $\mathrm{red}(\mathbf{N}^{q_2}(\mathbf{c}_2))=0$. It contradicts to Proposition \ref{Berk-dyn}. Now we may assume $\Gamma_1=\{0\}$. Let $\{M_t(z)\}$ be a family of scaling maps such that the induced map $\mathbf{M}(\xi_g)=0\vee\mathbf{s}$. Set $\widetilde N_t=M_t^{-1}\circ N_t\circ M_t$ and assume $\widetilde{N}_t\to\widetilde{N}$. Then $\mathrm{Hole}(\widetilde N)=\{0,\infty\}$. Set $\tilde z_k=M_{t_k}^{-1}(z_k)$ and denote $\widetilde\Gamma_1$ the limit of $\langle\tilde z_k\rangle$. By Proposition \ref{cycle123}, we have $\{0\}\subsetneqq\widetilde \Gamma_1$. Note $0$ is an attracting fixed point of $\mathrm{red}(\mathbf{M}^{-1}\circ\mathbf{N}\circ\mathbf{M})$. Thus by Proposition \ref{Fatou-size}, there exist $0\le q_1\le n_1-1$ and a critical point $\tilde{c}(t_k)$ of $\widetilde N_{t_k}$ such that $\widetilde N_{t_k}^{q_1}(\tilde c(t_k))$ converges to $0$. Equivalently, there is $i\in\{1,2\}$ such that $|\mathbf{N}^{q_1}(\mathbf{c}_i)|<|\mathbf{s}|$. It also contradicts to Proposition \ref{Berk-dyn}.\par
If $\mathbf{N}$ is of type $3b$, by Propositions \ref{cycle123}, we know $0\in\Gamma_1\cap\Gamma_2$. If $\{0\}\subsetneqq\Gamma_1$ and  $\{0\}\subsetneqq\Gamma_2$, by Proposition \ref{Fatou-size}, there exists $0\le q_1\le n_1-1$ and $0\le q_2\le n_2-1$  such that $N_{t_k}^{q_1}(c_1(t_k))$ converges to $0$ and $N_{t_k}^{q_2}(c_2(t_k))$ converges to $0$.  It follows that for the induced map $\mathbf{N}$, the reductions $\mathrm{red}(\mathbf{N}^{q_1}(\mathbf{c}_1))=0$ and $\mathrm{red}(\mathbf{N}^{q_2}(\mathbf{c}_2))=0$. It contradicts to Proposition \ref{Berk-dyn}. If $\Gamma_1=\{0\}$,  then by Propositions \ref{cycle123}, we have $\{0\}\subsetneqq\Gamma_2$. By Propositions \ref{Fatou-size}, \ref{critical-projection} and \ref{Berk-dyn}, we know $|\mathbf{c}_2|\le|\mathbf{r}|$. Let $\{L_t(z)\}$ be a family of affine maps such that the induced map $\mathbf{L}(\xi_g)=0\vee\mathbf{r}$. Suppose $\langle L^{-1}_{t_k}(z_k)\rangle$ converges to $\widetilde\Gamma_1$. Then $\widetilde\Gamma_1$ is a nonrepelling cycle of $\mathrm{red}(\mathbf{L}^{-1}\circ\mathbf{N}\circ\mathbf{L})$.  However, $\mathrm{red}(\mathbf{L}^{-1}\circ\mathbf{N}\circ\mathbf{L})$ is a cubic rational map with three fixed critical points and one critical points mapping to the repelling fixed point under iterate. It contradicts with the FSI.
\end{proof}

\section{Proof of Theorem \ref{accessible-bounded}}\label{accessible}
In this section, we prove Theorem \ref{accessible-bounded}.  

\begin{proof}
Let $\widetilde{H}$ be the preimage of $\mathcal{H}$ in $\mathrm{Rat}_4$. Suppose $[N]\in \partial\mathcal{H}$ is tame. By the definition of tame, there exists $N_{t_k}\in\widetilde{H}$ for some holomorphic family $N_t$ of Newton maps. By changing this parameterization by replacing $t$ with some power $t^l$ to account for the fact that the roots are not functions of $t$ a priori, we may find a degenerating holomorphic family $N_t=N_{\{a(t),b(t),c(t),d(t)\}} \to N:=H_N\hat{N}$ such that $N_{t_k} \in \widetilde{H}$ for some sequence $t_k \to 0$. 

By Lemma \ref{types}, by replacing $N_t$ with an affine conjugate family $M_t \circ N_t \circ M_t^{-1}$, we may and do assume the induced map $\mathbf{N}$ of $\{N_t\}$ is of type $1, 2, 3a,$ or $3b$.  

We first assume the following result, which is an analogue of Proposition \ref{intersection}.
\begin{proposition}\label{accessible-intersection}
Let $N_t, N_{t_k}, N$ be as above. Let $\langle z_k\rangle$ be the non-fixed attracting cycle of $N_{t_k}$ and let $\Gamma$ be the limit of $\langle z_k\rangle$. Then $\Gamma\cap\mathrm{Hole}(N)\not=\emptyset$.
\end{proposition}
Assuming Proposition \ref{accessible-intersection}, we continue with the proof. Let $n_1$ be the period of the non-fixed attracting cycle and let $n_2$ be the smallest nonnegative integer such that the $n_2$-th forward images of both two additional critical points $c_1(t_k)$ and $c_2(t_k)$ lie in the immediate basin of the non-fixed attracting cycle. Set $n=n_1+n_2$. If the induced map $\mathbf{N}$ is of type $1, 2,$ or $ 3a$, replace $\Gamma_1$ by $\Gamma$, the conclusions $(1)-(3)$ of Proposition \ref{cycle123} still hold. Then considering $M_t^{-1}\circ N_t\circ M_t$ if necessary, by Proposition \ref{Fatou-size}, we know there exist $0\le q_1\le n-1$ and $0\le q_2\le n-1$ such that $N_{t_k}^{q_1}(c_1(t_k))$ converges to $0$ and $N_{t_k}^{q_2}(c_2(t_k))$ converges to the holes of $N$. It follows that for the induced map $\mathbf{N}$, the reductions $\mathrm{red}(\mathbf{N}^{q_1}(\mathbf{c}_1))$ and $\mathrm{red}(\mathbf{N}^{q_2}(\mathbf{c}_2))$ are in $\mathrm{Hole}(N)$. It contradicts to Proposition \ref{Berk-dyn}. Thus $\mathbf{N}$ is of type 3b. Again, by Propositions \ref{Fatou-size} and \ref{Berk-dyn}, we know $\Gamma=\{0\}$. Let $\{L_t(z)\}$ be a family of affine maps such that the induced map $\mathbf{L}(\xi_g)=0\vee\mathbf{r}$ and let $\widetilde{\Gamma}$ be the limit of $\langle L_{t_k}^{-1}(z_k)\rangle$. Then $\widetilde{\Gamma}$ is a nonrepelling cycle of $\mathrm{red}(\mathbf{L}^{-1}\circ\mathbf{N}\circ\mathbf{L})$. Proposition \ref{number} implies that $\widetilde{\Gamma}$ is not parabolic-attracting or parabolic-indifferent. 
\end{proof}
Now we prove Proposition \ref{accessible-intersection}; the proof uses complex analytical estimates. 
\begin{proof}[Proof of Proposition \ref{accessible-intersection}]
Suppose to the contrary that $\Gamma\cap\mathrm{Hole}(N)=\emptyset$ and write $N=H_N\widehat N$. Then $\Gamma$ is a nonrepelling cycle of $\widehat{N}$. The FSI implies the induced map $\mathbf{N}$ is not of type $2,3a$, or $3b$.\par
We may therefore assume that the case that $\mathbf{N}$ is of type $1$. Let $n_1$ be the period of the non-fixed attracting cycle and let $n_2$ be the smallest nonnegative integer such that the $n_2$-th forward images of both two additional critical points $c_1(t_k)$ and $c_2(t_k)$ lie in the immediate basin of the non-fixed attracting cycle. Assume $z_k^{(i)}\to z^{(i)}$ and let $\Omega_k^{(i)}$ be the Fatou component containing $z_k^{(i)}$. \par
We claim that there is a neighborhood $V$ of $0$ such that $\Omega_k^{(i)}\cap V=\emptyset$ for sufficiently large $k$ and all $0\le i\le n_1-1$. Indeed, consider a small neighborhood $V$ of $0$ in the immediate basin of the fixed point $0$ for the map $\widehat{N}$. We can choose $V$ small enough such that $z^{(i)}\not\in\mathcal{D}$ for all $0\le i\le n_1-1$, where $\mathcal{D}$ is the component of $\widehat{N}^{-n_1}(V)$ containing $0$. Let $\mathcal{B}$ be a small neighborhood of $\partial{\mathcal{D}}$ such that $V\subset\mathcal{D}\setminus\mathcal{B}$ and $z^{(i)}\not\in\mathcal{B}$ for $0\le i\le n_1-1$. Note $\mathrm{Hole}(N^{n_1})\cap(\mathcal{D}\setminus\mathcal{B})=\{0\}$. Then for sufficiently large $k$, we have $N_{t_k}^{n_1}(\mathcal{D}\setminus\mathcal{B})\subset V$. Let $U_{k,0}^{(i)}$ be a small neighborhood of $z_k^{(i)}$ such that $\overline{N_{t_k}^{n_1}(U_{k,0}^{(i)})}\subset U_{k,0}^{(i)}$ and the boundary $\partial U_{k,0}^{(i)}$ is a simple closed curved containing no iterated forward images of critical points. Let $U_{k,j}^{(i)}$ be the component of $N_{t_k}^{-jn_1}(U_{k,0}^{(i)})$ containing $z_k^{(i)}$. Then $U_{k,j}^{(i)}\subset U_{k,j+1}^{(i)}$ and $\Omega_k^{(i)}=\cup_{j\ge 0} U_{k,j}^{(i)}$. If there is $j\ge 1$ such that $U_{k,j}^{(i)}\cap(\mathcal{D}\setminus\mathcal{B})\not=\emptyset$, let $j_0$ be the smallest such $j$. Then we have $N_{t_k}^{n_1}(U_{k,j_0}^{(i)}\cap(\mathcal{D}\setminus\mathcal{B}))\subset N_{t_k}^{n_1}(U_{k,j_0}^{(i)})=U_{k,j_0-1}^{(i)}\subset\widehat{\mathbb{C}}\setminus(\mathcal{D}\setminus\mathcal{B})$ and hence $N_{t_k}^{n_1}(U_{k,j_0}^{(i)}\cap\mathcal{D}\setminus\mathcal{B})\subset\widehat{\mathbb{C}}\setminus V$. On the other hand, we have $N_{t_k}^{n_1}(U_{k,j_0}^{(i)}\cap(\mathcal{D}\setminus\mathcal{B}))\subset N_{t_k}^{n_1}(\mathcal{D}\setminus\mathcal{B})\subset V$. It is a contradiction. So $U_{k,j}^{(i)}\cap(\mathcal{D}\setminus\mathcal{B})=\emptyset$ for all $j\ge 0$.  It follows that $\Omega_k^{(i)}\cap(\mathcal{D}\setminus\mathcal{B})=\emptyset$. Therefore, $\Omega_k^{(i)}\cap V=\emptyset$.\par
By Proposition \ref{Berk-dyn}, there exists at least one critical point, say $\mathbf{c}_1$, of the induced map $\mathbf{N}$ such that $\mathrm{red}(\mathbf{N}^\ell(\mathbf{c}_1))\to 0$ as $\ell\to\infty$. Then there exists $n_3\ge n_2$ such that $N_{t_k}^{n_3}(c_1(t_k))\in V$ for sufficiently large $k$. However, $N_{t_k}^{n_3}(c_1(t_k))\in\Omega_{k}^{(i_0)}$ for some $0\le i_0\le n_1-1$. It is impossible since $\Omega_{k}^{(i_0)}\cap V=\emptyset$ for sufficiently large $k$. Therefore, we have $\Gamma\cap\mathrm{Hole}(N)\not=\emptyset$. 
\end{proof}

\section{Immediate escape components are unbounded}\label{qc}
In this section, we prove Theorem \ref{escape-unbounded}, which asserts that the immediate escape regions in $\mathrm{nm}_d$ are unbounded.
\begin{proof}[Proof of Theorem \ref{escape-unbounded}]
Suppose $f$ is a hyperbolic Newton map for which some additional critical point lies in the immediate basin $\Omega$ of some root; we may assume this root is the origin.  Applying the quasiconformal surgery arguments of McMullen  \cite{McMullen88} to the backward (equivalently, grand) orbit of the basin $\Omega$, we may alter $f$ within $\Omega$ to assume that (i) each additional critical point $c'$ of $f$ lying in some preimage $\Omega' \neq \Omega$ is in the backward orbit of the origin, (ii) each critical point not in the grand orbit of $\Omega$ has finite forward orbit, and (iii) the dynamics of $f$ on the basin $\Omega$ is holomorphically conjugate to the Blaschke product 
\[ B_a(w)=-w^k\frac{w-a}{1-\overline{a}w}, 0<|a|<1, k \geq 2.\]
Letting $a \uparrow 1$ along the reals then uniquely determines a real-analytic one-parameter family $a \mapsto f_a$ of Newton maps in $\mathcal{H}$.  As $a \uparrow 1$, direct calculations show (i) the unique non-fixed critical point $x_a$ of $B_a$ tends to $1$; (ii) the ``ray segment'' $[0, x_a]$ is forward-invariant under $B_a$; (iii) the hyperbolic distance $d(x_a, B_a(x_a))$ remains bounded; (iv) $B_a$ degenerates, converging locally uniformly on the complement of the point $w=1$ to the lower degree map $w \mapsto w^k$. 
We will show, indirectly, that the family $f_a$ tends to infinity in $\mathcal{H}$ as $a \uparrow 1$.  \par
Suppose otherwise. Then for some sequence $a_n \uparrow 1$, the family $f_n:=f_{a_n}$ converges uniformly to some limiting Newton map $f$. The argument below is a special case of a general method used to show convergence of certain ``rational angle'' (in this case, the angle is $0/1$) parameter rays; see \cite{Petersen00}. Our setting is combinatorially simpler, and we use the language of Carath\'eodory convergence; see \cite{McMullen94}.\par
Let $\Omega_n$ denote the immediate basin of the origin for $f_n$, and $c_n \neq 0$ the other critical point of $f_n$ in $\Omega_n$; in the Blaschke product model, this critical point is on the positive real axis, between $0$ and $a$.  Since $f_n \to f$, 
(i) the origin is again a superattracting fixed-point of $f$ of local degree $k$; let $\Omega$ denote its immediate basin, and (ii) the critical points $c_n$ converge to a (necessarily simple) critical point $c$ of $f$.\par
Property (i) of the previous paragraph implies the pointed basins $(\Omega_n, 0) \to (\Omega, 0)$ converge in the Carath\'eodory topology. Indeed, for any compact subset $K\subset U$, there exists a disk $B(0,r)$ such that $K\subset B(0,r)\subset\Omega\cap\Omega_n$ for sufficiently large $n$, and for any open connected subset $U\subset\Omega_n$ containing $0$ for infinitely many $n$, there are no repelling periodic points  of $f_n$ in $U$,  which implies $U\subset\Omega$.  Since the hyperbolic distance in $\Omega_n$ between $c_n$ and $0$ tends to infinity, $c \not\in \Omega$. Consider the pointed basins $(\Omega_n, c_n)$.  There are now two cases to consider. If no Carath\'eodory limit point exists, then the Euclidean distance between $c_n$ and $\partial \Omega_n$ tends to zero \cite[Theorem 5.2]{McMullen94}. Since the hyperbolic distance in $\Omega_n$ between $c_n$ and $f_n(c_n)$ is uniformly bounded and the hyperbolic distance is comparable to the ``$1/d$''-metric \cite[Theorem 2.3]{McMullen94}, we conclude $|c_n-f_n(c_n)| \to 0$, hence that $f(c)=c \neq 0$, which is impossible since $f_n^k(c_n) \to 0$ as $k \to \infty$.  We condlude there is a limit point $(\Omega', c)$ of $(\Omega_n, c_n)$. Since the hyperbolic distance in $\Omega_n$ between $c_n$ and $f_n(c_n)$ is uniformly bounded, the domain $\Omega'$ is a fixed simply-connected Fatou component of $f$ distinct from $\Omega$ mapping to itself by degree two. It therefore cannot be a Siegel disk or Hermann ring. It cannot be an attracting or superattracting basin since these are stable under perturbation. We conclude as in \cite{Petersen00} that $\Omega$ is an immediate basin of a parabolic fixed-point for $f$.\par
This, however, is impossible. Since $f$ is a Newton map, the point at infinity is the unique non-critical fixed-point: by the Holomorphic Fixed-Point Index Formula \cite[Theorem 12.4]{Milnor06B} it is repelling with multiplier $d/(d-1)$.
\end{proof}

\bibliographystyle{siam}
\bibliography{references}
\end{document}